\pdfoutput=1
\documentclass{amsart}
 \usepackage{booktabs}
\usepackage{amssymb}
\usepackage[utf8]{inputenc}
\usepackage{longtable, geometry}
\usepackage[english]{babel}
\makeindex

\usepackage{hyperref}
\hypersetup{
colorlinks=true,       % false: boxed links; true: colored links
    linkcolor=blue,          % color of internal links
    citecolor=blue,        % color of links to bibliography
    filecolor=blue,      % color of file links
    urlcolor=blue           % color of external links
}

\usepackage{rotating,tabularx}

\usepackage[all]{xy}
\usepackage{tikz}
\usepackage{todonotes}
\usepackage[alphabetic,backrefs,lite]{amsrefs}
\usepackage{enumerate}
\usepackage{verbatim}
\usepackage{multirow}
\usepackage{pgf,tikz}
\usepackage{float}
\usepackage{rotating}
\usetikzlibrary{arrows}
\usepackage{listings}

\CompileMatrices

\newtheorem{theorem}{Theorem}[section]
\newtheorem{lemma}[theorem]{Lemma}

\newtheorem{proposition}[theorem]{Proposition}
\newtheorem{corollary}[theorem]{Corollary}
\theoremstyle{definition}

\newtheorem{example}[theorem]{Example}

\newtheorem{remark}[theorem]{Remark}

\def\cc{{\mathbb C}}

\def\zz{{\mathbb Z}}

\def\nn{{\mathbb N}}
\def\qq{{\mathbb Q}}
\def\pp{{\mathbb P}}
\def\ff{{\mathbb F}}
\def\Osh{{\mathcal O}}

\def\Eff{\operatorname{Eff}}

\def\Cl{\operatorname{Cl}}

\def\NS{\operatorname{NS}}
\def\Nef{\operatorname{Nef}}

\def\BNef{\operatorname{BNef}}
\def\BEff{\operatorname{BEff}}
\def\rk{\mbox{rank}}
\renewcommand{\div}{\operatorname{div}}
\newcommand{\oo}{\mathcal{O}}

\begin{document}

\title[Mori dream K3 surfaces of Picard number four]{Mori dream K3 surfaces of Picard number four:\\ projective models and Cox rings}

\author{Michela Artebani}
\address{
Departamento de Matem\'atica, \newline
Universidad de Concepci\'on, \newline
Casilla 160-C,
Concepci\'on, Chile}
\email{martebani@udec.cl}

\author{Claudia Correa Deisler}
\address{
Departamento de Matem\'atica, \newline
Universidad de Concepci\'on, \newline
Casilla 160-C,
Concepci\'on, Chile}
\email{claucorrea@udec.cl}

\author{Xavier Roulleau}
\address{
Aix-Marseille Universit\'e,\newline
CNRS, Centrale Marseille \newline
 I2M UMR 7373,\newline
13453 Marseille, France }
\email{Xavier.Roulleau@univ-amu.fr}

\thanks{The authors have been partially 
supported by Proyecto FONDECYT Regular 
N. 1160897 and by Proyecto Anillo ACT 1415 PIA CONICYT (first and second author).
The second author has been supported by CONICYT PCHA/DoctoradoNacional/2012/21120687.}

\begin{abstract}
In this paper we study the geometry of the $14$ families of 
K3 surfaces of Picard number four 
with finite automorphism group, 
whose N\'eron-Severi lattices have been classified by  \`E.~B. Vinberg. 
We provide projective models, 
we identify the degrees of a generating set of the Cox ring 
and in some cases we prove the unirationality of the associated moduli space.
\end{abstract}
\maketitle

\tableofcontents

\section*{Introduction} 
It follows from the global Torelli theorem that the automorphism group 
of a complex K3 surface $X$ is finite if and only if the quotient group 
${\rm O}(S)/W^{(2)}(S)$ 
is finite, where $S$ is the N\'eron-Severi lattice of $X$, ${\rm O}(S)$ is its isometry group and 
${\rm W^{(2)}}(S)$ is the subgroup generated by reflections relative to classes with self-intersection $-2$.
The hyperbolic lattices with the latter property are called  $2$-reflective 
and there are a finite number of them for each rank $\geq 3$.
Their classification is due to V.V. Nikulin \cites{VN1,VN,VN2} for rank $3$ and $\geq 5$,
and to \`E.\,B. Vinberg for rank $4$ \cite{EV1}. 
Complex K3 surfaces with finite automorphism group and Picard number at least three 
can be also characterized as those 
K3 surfaces having  finitely many $(-2)$-curves,
and in such case the classes of such curves generate the effective cone of the surface.

%have been classified 
%in \cite{VN1,VN,VN2,P.S, EV1} according to their N\'eron-Severi lattice.  
%As a co
%For Picard number $\geq 3$ there is a finite number of families with such property.
In this paper we study the geometry of K3 surfaces with finite automorphism group and 
having Picard number $4$: we compute the effective cone  and
the nef cone (Theorem \ref{main-eff}), we find projective models and we study their Cox ring.
In particular, we prove the following result, for more details see Propositions \ref{gen v1},  \ref{gen v2},  \ref{gen v3},  \ref{gen v4}, \ref{gen v5},  \ref{gen v6},  \ref{gen v7},  \ref{gen v8},  \ref{gen v9},  \ref{gen v10},  \ref{gen v11},  \ref{gen v12},  \ref{gen v13},  \ref{gen v14}.

In what follows we will denote by $\mathcal F_i$, $i=1,\dots,14$, the family of Mori dream K3 surfaces whose N\'eron-Severi lattice 
is isometric to the lattice $V_i$ defined in Theorem \ref{classification} and by $\mathcal M_i$ the associated moduli space.

\begin{theorem} Let $X$ be a complex K3 surface with Picard number four and with finite automorphism group. 
Then up to isomorphism equivalence $X$ belongs to one of the following families:
\begin{enumerate}
\item[$\mathcal F_1$:] complete intersections of three quadrics in $\pp^5$ having three nodes;
\item[$\mathcal F_2$:] minimal resolutions of quartic surfaces in $\pp^3$ with a node and two hyperplane sections through the node which are union of two conics;
\item[$\mathcal F_3$:] minimal resolutions of double covers of $\pp^2$ branched along a plane sextic with a double point with two bitangent lines passing through the point;
\item[$\mathcal F_4$:] double covers of the Hirzebruch surface $\ff_4$ blown-up at two general points;
\item[$\mathcal F_5$:] minimal resolutions of double covers of $\pp^2$ branched along a plane sextic with three nodes;
\item[$\mathcal F_6$:] double covers of $\pp^2$ branched along a smooth plane sextic with two 3-tangent lines and two 6-tangent conics;
\item[$\mathcal F_7$:] smooth quartic surfaces in $\pp^3$ having four hyperplane sections which are the union of two conics;
\item[$\mathcal F_8$:] minimal resolutions of double covers of $\ff_3$ branched along the union of the smooth rational curve $E$  with $E^2=-3$ and   
a reduced curve $B$ intersecting $E$ at one point $p$ and 
the fiber of $\ff_3\to \pp^1$ through $p$ at one point with multiplicity two;
\item[$\mathcal F_9$:] minimal resolutions of double covers of 
$\pp^2$ branched along a plane sextic having two nodes and such that the line through the nodes is tangent to the sextic curve at one more point;
\item[$\mathcal F_{10}$:] smooth quartic 
 surfaces in $\pp^3$ having one hyperplane section which is the union of four lines;
\item[$\mathcal F_{11}$:] smooth quartic 
 surfaces in $\pp^3$ having three reducible hyperplane sections which are the union of two conics;
\item[$\mathcal F_{12}$:] double covers of $\pp^2$ branched along a smooth plane sextic with three $3$-tangent lines such that the cover is trivial over the union of the three lines;
\item[$\mathcal F_{13}$:] double covers of
$\pp^2$ branched along a smooth plane sextic with three $3$-tangent lines such that the cover is not trivial over the union of the three lines;
\item[$\mathcal F_{14}$:] double covers of
$\pp^2$ branched along a smooth plane sextic with one $3$-tangent line and three $6$-tangent conics.
\end{enumerate}
\end{theorem}

By \cite[Theorem 2.7 and Theorem 2.11]{A.H.L}  the Cox ring of a 
complex projective K3 surface is finitely 
generated if and only if its automorphism group is finite.
Using techniques developed in \cite{A.C.L}, 
we compute the degrees of a generating set for the
Cox rings of K3 surfaces with finite automorphism group 
of Picard number $4$. 
A similar description for the case of Picard number three 
has been given in \cite{A.C.L}. Our main result is the following.

%\michela{to be completed: what about stating a theorem here containing a short description of all projective models 
%and the unirationality results?}
%For K3 surfaces of Picard number four we obtain the following result,

\begin{theorem}\label{main-cox}
Let $X$ be a  Mori dream K3 surface of Picard number four 
such that $\NS(X)$ is not isometric to $V_{14}$.
The degrees of a set of generators of the Cox ring $R(X)$ are given in
Table \ref{TableGen rank4} and all degrees in the Table are necessary to generate 
$R(X)$, except eventually for those marked with a star.
In case $\NS(X)\cong V_{14}$ the degrees of a set of generators of $R(X)$  
are either classes of $(-2)$-curves or sums of at most three elements in the Hilbert 
basis of the nef cone. Moreover, the degrees given in Table \ref{Table degree generator F14}
 are necessary to generate $R(X)$.
\end{theorem}

Moreover, for a very general K3 surface in the families $\mathcal F_i$ with $i\in \{4,5,6,7,10,12\}$ 
we provide a presentation for the Cox ring.

As a consequence of the first Theorem we also obtain the following result, see Corollary \ref{uni1}, \ref{uni3}, \ref{uni4}, \ref{uni5}, \ref{uni6}, \ref{uni7}, \ref{uni8}, \ref{uni9}, \ref{uni10}, \ref{uni11} and \ref{uni12}.

\begin{theorem} 
The moduli spaces  $\mathcal M_i$ with $i\in \{1,3,4,5,6,7, 8,9,10,11,12\}$ are unirational.
Moreover, $\mathcal M_3$ is rational.
%Let $X$ be an algebraic K3 surface with Picard number four and finite automorphism group.
% If $\NS(X)$ is isometric to one of the following lattices:
%$$
%V_1=(8)\oplus 3A_1,
%\qquad
%V_{3+k}=U(k)\oplus 2A_1,\ k\in\{1,2,3,4\},
%$$
%$$
%V_{7+k}=U(k)\oplus A_2, \ k \in\{1,2,3\},
%\qquad
%V_{11}=U(6)\oplus A_2,
%\qquad
%%\tiny
%V_{12}=\left[
%\begin{array}{rr}
%0&-3
%\\
%-3&2
%\end{array}
%\right]
%\oplus A_2.
%$$
%Then the moduli space of K3 surfaces with $\NS(X)$ in the list above is unirational.
\end{theorem}

\section{Preliminaries}
We will work over the field $\cc$ of complex numbers.

\subsection{Linear systems of K3 surfaces}

To be self-contained, let us recall the following results of Saint-Donat \cite{SD}. For a divisor $D$ such that the linear system $|D|$ is base point free, 
we denote by $\varphi_{D}$ the morphism associated to $|D|$.
 
\begin{theorem}\label{bl}
\label{thm:SaintDonat-1} Let $D$ be a divisor on a K3 surface $X$.\\
a) if D is effective, non-zero, and $D^{2}=0$, then $D\sim aF$ where
$|F|$ is a free elliptic pencil and $a\in\nn$.\\
b) If $D$ is big and nef, then $h^{0}(D)=2+\frac{1}{2}D^{2}$, and
either $|D|$ has no fixed part, or $D\sim aF+E$ where $F$ is
a smooth elliptic curve and $E$ is a $(-2)$-curve such that $F\cdot E=1$.\\
c) If $D$ is big and nef and $|D|$ has no fixed part, then $|D|$
is base point free and $\varphi_{D}$ is either $2$ to $1$ onto
its image (hyperelliptic case) or it maps $X$ birationally onto its
image, contracting the $(-2)$-curves $\Gamma$ such that $D\cdot \Gamma=0$ to
singularities of type $ADE$. 
\end{theorem}

\begin{corollary}\label{section}
Let $X$ be a K3 surface such that none of its elliptic fibrations has sections 
and $D$ be an effective non-zero nef divisor of $X$, then $|D|$ is base point free.
\end{corollary}
\begin{proof}
Assume that $D$ is a nef divisor with non-empty base locus. By Theorem \ref{bl} $D$ is linearly equivalent to a sum involving two smooth curves $F, E$ such that $F$ is a smooth elliptic curve and $E$ a $(-2)$-curve with $F\cdot E=1$. The morphism associated to $|F|$ defines an elliptic fibration $\varphi_F:X\to \pp^1$ (see \cite{SD}) and $E$ is a section of it.
\end{proof}

\begin{theorem}\label{hyp}
\label{thm:SaintDonat-2}(See \cite[Proposition 5.7]{SD}
and its proof). Let $|D|$ be a complete linear system without fixed components 
on a K3 surface $X$.  Then $|D|$ is
hyperelliptic if and only if  one of the following cases occurs: \\
a) $D^{2}\geq 4$ and there is a smooth elliptic curve $F$ such that
$F\cdot D=2$;\\
b) $D^{2}\geq 4$ and there is an irreducible curve $D'$ with $D'^{2}=2$
and $D\sim 2D'$ (thus in that case $D^{2}=8$).\\ 
%In this case the image ofthe map $\varphi_{D}$ is the Veronese surface.\\
c) $D^{2}=2$, and in this case $\varphi_{D}$ is a double cover of $\pp^{2}$.
\end{theorem}

\begin{remark}
For brevity, we will say that a divisor $D$ is base point free or
is hyperelliptic if the linear system $|D|$ is.
\end{remark}

\subsection{Vinberg's classification}
We now state the classification of K3 surfaces 
of Picard number four with finite automorphism group due to \`E.\,B. Vinberg  \cite[Theorem 1]{EV1}.
As usual we denote by $A_m, D_n, E_6, E_7, E_8$ the negative definite lattices that
correspond to the root systems of the same types ($m\geq 1, n\geq 4$), 
and by $U$ the lattice with Gram matrix 
$\left(\begin{array}{cc}
0 & 1\\
1 & 0
\end{array}\right)$. Moreover, $L_1\oplus L_2$ denotes the direct sum of the lattices $L_1$ and $L_2$, 
$nL$ ($n\in \nn$) denotes the direct sum of $n$ copies of the lattice $L$ and  $L(k)$ ($k\in \zz$)
is the lattice with the same underlying group as $L$, but with Gram matrix multiplied by $k$.

%The following result follows Vinberg's classification of  $2$-reflective even hyperbolic lattices 
%of rank $4$ \cite[Theorem 1]{EV1}.

\begin{theorem}\label{vinberg}
\label{classification}
Let $X$ be an algebraic complex K3 surface with 
Picard number $\rho(X)=4$ and finite automorphism group. 
Then $\NS(X)$ is isometric 
to one of the following $14$
lattices:
$$
V_1=(8)\oplus 3A_1,
\qquad
V_2=(-4)\oplus (4)\oplus A_2,
\qquad
V_3=(4)\oplus A_3,
\qquad
V_{3+k}=U(k)\oplus 2A_1,\ k\in\{1,2,3,4\},
$$
$$
V_{7+k}=U(k)\oplus A_2, \ k \in\{1,2,3\},
\qquad
V_{11}=U(6)\oplus A_2,
\qquad
%\tiny
V_{12}=\left[
\begin{array}{rr}
0&-3
\\
-3&2
\end{array}
\right]
\oplus A_2,
$$
$$
V_{13}=\left[
\begin{array}{rrrr}
2&-1&-1&-1
\\
-1&-2&0&0
\\
-1&0&-2&0
\\
-1&0&0&-2
\end{array}\right],
\qquad
V_{14}=
\left[
\begin{array}{rrrr}
12&-2&0&0
\\
-2&-2&-1&0
\\
0&-1&-2&-1
\\
0&0&-1&-2
\end{array}
\right].
$$
\end{theorem}

\subsection{Effective and nef cones}\label{eff-nef}
Let $X$ be a K3 surface and $\Eff(X)\subseteq \Cl(X)_{\qq}\cong {\rm NS}(X)_{\qq}$ be its {\em effective cone}.
If $\rho(X)\geq 3$ 
and $\Eff(X)$ is polyhedral, then its extremal rays are generated by classes of $(-2)$-curves (see \cite[Proposition 2.13]{A.H.L}).
 In order to identify the classes of all $(-2)$-curves we use the following algorithm,
 known as  ``Vinberg's algorithm'' \cite{Vinberg}:
%which has been implemented in the Magma program \text{Find-2.m} (see section \ref{magma}): 

\begin{enumerate}[$\bullet$]
\item Fix a class $\alpha\in \Cl(X)$ with $\alpha^2>0$.
\item Find all classes $w\in \Cl(X)\cap \alpha^{\perp}$ with self-intersection $-2$ 
(there are finitely many of them since the restriction of $Q$ to $\alpha^{\perp}$  
is negative definite) and let $L$ be a list of such classes, which is a root system.
\item If $L$ is not empty, fix a set $L^+\subseteq L$ of positive roots  as follows:
choose randomly an integral combination $H$ of  the vectors in $L$
having non zero intersection with all of them and let $L^+$ be the set of vectors in $L$
having positive intersection with $H$.
\item  Construct the list $R_0$ of simple roots in $L^+$ inductively as follows: 
let $L^+_0\subseteq L^+$ be the set of vectors having minimal intersection $m_0$ with $H$ 
and, once $L^+_i$ is given for some $i\geq 0$, define $L^{+}_{i+1}\subseteq L^+$ as the  
set of vectors having intersection $m_0+i+1$ with $H$ and non-negative intersection with all vectors in $L^+_k$ for $0\leq k\leq i$.
The process stops when $m_0+n=\max\{v\cdot H: v \in L^+\}$ and the set of simple roots in $L^+$ is 
$R_0:=\cup_{0\leq i\leq n}L^+_{i}$.
\item Construct a set of fundamental roots for $\Cl(X)$ inductively as follows:  let $R_0$ be as in the previous item ($=\emptyset$ if $L$ is empty) 
and define $R_{i+1}$ as the union of $R_i$ with the set of classes $w\in \Cl(X)$ such that $w^2=-2$, $w\cdot \alpha=i+1$ and 
having non-negative intersection with all the elements of $R_i$.
\item The set $R_n$ is a set of fundamental roots of $\Cl(X)$ if the following property holds for the 
convex polyhedral cone $\mathcal C$ generated by the vectors in $R_n$: the intersection matrix of the vectors 
generating any facet of $\mathcal C$ is negative semidefinite. 
\end{enumerate}

%\begin{enumerate}[$\bullet$]
%\item we find a diagonal matrix $D\in {\rm GL}(n,\qq)$ with $D_{1,1}>0$ and $D_{i,i}<0$ for $i\not=0$ 
%and $B \in {\rm GL}(n,\qq)$ such that $D=BQB^T$;
%\item  we find the list $L$ of all vectors $v=B^Ty\in \zz^n$ 
%such that $v^2=-2$ and $y_1=0$ (this gives a root system);
%\item if $\#L\geq 2$,   
%after choosing randomly an integral combination $H$ of  the vectors in $L$
%having non zero intersection with all of them, we find the list  $L_0$ of simple roots 
%having positive intersection with $H$;
%\item  for $i\geq 1$ we find inductively the list $L_i$ of all vectors $v=B^Ty\in \zz^n$ such that $v^2=-2$ and 
% $y_1=\frac{i}{d}$ with $d=|{\rm Det}(B)|$ having non-negative intersection with all vectors in $L_{i-1}$ 
% and at each step we define $L$ as the union of $L_{i-1}$ with $L_i$;
% \item we compute the cone $C$ generated by
%the vectors in $L$. If the intersection matrix of the vectors
%generating $F$ is negative semidefinite for any facet $F$ of $C$, then 
%$L$ is the list of classes of all $(-2)$-curves.
%\end{enumerate}

We recall that the {\em nef cone} $\Nef(X)\subseteq \Cl(X)_{\qq}$ 
of a smooth surface is the dual of its effective cone with respect 
to the intersection form.
In what follows we will denote by $\BEff(X)$ and $\BNef(X)$ 
the {\em Hilbert bases} of $\Eff(X)$ and $\Nef(X)$, i.e. 
the unique minimal generating sets of $\Eff(X)\cap \Cl(X)$ 
and $\Nef(X)\cap \Cl(X)$ (see for example \cite[Definition 7.17]{M.S}). 

The previous algorithm has been implemented in a Magma \cite{B.C.P} program 
which is described in \cite{A.C.L} and included as ancillary file in the arXiv version of the same paper:
 \begin{center}
 \href{https://arxiv.org/abs/1909.01267}{https://arxiv.org/abs/1909.01267}.
\end{center}

\subsection{Cox rings}\label{cox}
We now recall some preliminary facts on Cox rings of surfaces, based on \cite{A.D.H.L}.
The \emph{Cox ring} of a smooth projective complex variety $X$ 
with finitely generated and free divisor class group $\Cl(X)$ is defined as
\[
R(X):=\bigoplus_{D\in K}\Gamma(X,\oo_X(D)),
\]
where $K\subseteq \rm{WDiv}(X)$ is a subgroup of the group 
of Weil divisors such that $K\to \Cl(X),\ D\mapsto [D]$ 
is an isomorphism. 
 Observe that $R(X)$ is a $K$-graded algebra over $\cc$. 
 The variety $X$ is called {\em Mori dream space} if $R(X)$ is a finitely generated 
 complex algebra.
 
 To simplify notation, given $w\in \Cl(X)$, 
 we will denote by $R(X)_{w}$ the vector space 
 $\Gamma(X,\oo_X(D))$, where $D\in K$ has  $[D]=w$.
 Moreover, in this case we will say that $f\in \Gamma(X,\oo_X(D))$ 
 is homogeneous of degree $w$.
%The variety $X$  is called
% \emph{Mori dream space} if it has finitely generated 
% Cox ring $R(X)$. 
Given a set of homogenous generators $\{f_i: i\in I\}$ of $R(X)$,
 it can be easily proved that their degrees $\{e_i: i\in I\}\subset K$ 
 generate the effective cone of $X$  (see \cite[Proposition 2.1]{A.H.L}).
 In particular, the effective cone of a Mori dream space is polyhedral. 
 Moreover, if $D$ is an effective divisor whose class 
 belongs to the Hilbert basis of the effective cone,
then any generating set of $R(X)$ contains en element of degree $[D]$.

% whose homogeneous 
%i.e. it has a direct sum decomposition into complex vector spaces 
%\[
%R(X)_D:=H^0(X,\oo_X(D)),
%\] 
%where $D\in K$, such that
%$ R(X)_{D_1}\cdot R(X)_{D_2}\subseteq R(X)_{D_1+D_2}.$
%We will say that $f\in R(X)$ is {\em homogeneous} 
%if it belongs to $R(X)_D$ for some $D\in K$ and in this case 
%we define its {\em degree} to be ${\rm deg}(f) = [D]$.
%A set of homogeneous generators $\{f_i: i \in I\}$ of $R (X)$ is a {\em minimal generating set} 
%when each $f_i$ can't be expressed as a polynomial in the remaining elements $f_j$. 
%Moreover, we say that $R (X)$ {\em has a generator in degree} $w\in \Cl (X)$ 
%if each minimal generating set of $R (X)$ contains a nontrivial element of $R (X)_D$, where $[D]=w$.\\
%In the following, given a divisor $A$, we denote by $R (X)_A$ the vector space $R (X)_D$ 
%where $D\in K$ and $[D] = [A]$.

%The present paper deals with Mori dream K3 surfaces.
The following theorem \cite[Theorem 5.1.5.1]{A.D.H.L}, together with the classification 
of K3 surfaces with finite automorphism group (see, \cite{VN1,VN,VN2,P.S,EV1}) 
provides a complete classification of K3 surfaces with finitely generated Cox ring.

 \begin{theorem}\label{Mori dream}
Let $X$ be an algebraic K3 surface. Then the following statements are equivalent.
\begin{enumerate}
\item $X$ is a Mori dream surface.
\item The effective cone $\Eff(X)\subseteq \Cl_{\qq}(X)$ is polyhedral.
\item The automorphism group of $X$ is finite.
\end{enumerate}
\end{theorem}

In \cite{A.C.L} we proved the following result on 
the degrees of a generating 
set of the Cox ring of a K3 surface.
We say that $R(X)$ {\em has a generator in degree} $w$ 
when each minimal set of homogeneous generators 
of $R(X)$ contains an element of degree $w$.

\begin{theorem}\label{gen}
$($\cite[Theorem 2.7]{A.C.L}$)$
Let $X$ be a smooth projective K3 surface over $\cc$.
Then the degrees of a minimal set of generators of its Cox ring $R(X)$ are either  
\begin{enumerate}
\item classes of $(-2)$-curves, 
\item  sums of at most three elements 
 of the Hilbert basis of the nef cone (allowing repetitions), 
 \item classes of divisors of the form $2(F+F')$, where $F,F'$ are smooth elliptic curves with $F\cdot F'=2$.
 \end{enumerate}
 \end{theorem}

Moreover, we developed different techniques to 
show that $R(X)$ has no generators in a certain degree 
(see \cite[Corollary 2.2, Corollary 2.4]{A.C.L}).  
Given $f\in \cc(X)^*$, we denote by $\div_E(f)$ the divisor $\div(f)+E$.

\begin{theorem}\label{teokoszul}
 Let $X$ be a smooth projective variety over $\cc$, $E_1,E_2,E_3$ 
 be effective divisors of $X$, $f_{i}\in H^0(X, E_i)$, $i=1,2,3$ 
 and $D\in {\rm WDiv}(X)$.
 \begin{enumerate}
 \item If $\cap_{i=1}^2\div_{E_i}(f_i)=\emptyset$ and  $h^1(X,D-E_1-E_2)=0$, 
then the following morphism is surjective
 \[
 H^0(X,D-E_1)\oplus H^0(X,D-E_2)\to H^0(X,D),\quad (g_1, g_2)\mapsto f_1g_1 + f_2g_2.
 \]
 \item If $\cap_{i=1}^3\div_{E_i}(f_i)=\emptyset$,  $h^1(X, D-E_i-E_j)=0$ 
 for all distinct $i,j\in \{1,2,3\}$ and $h^2(X, D-E_1-E_2-E_3)=0$,
 then the following morphism is surjective
\[
\bigoplus_{i=1}^3 H^0(X, D-E_i)\to H^0(X,D),\ (g_1,g_2,g_3)\mapsto f_1g_1+f_2g_2+f_3g_3.
\]
\end{enumerate}
\end{theorem}

The following result  shows that $R(X)$ is not generated in degrees which are 
sums of a very ample class and the class of an elliptic fibration under certain conditions.

 \begin{proposition}\label{va}
  $($\cite[Lemma 2.8]{A.C.L}$)$
Let $X$ be a K3 surface and $D=F+D'$ a nef divisor, where $F$ is nef with $F^2=0$ 
and $D'$ is very ample. Assume that $F\sim E_1+E_2$, where $E_1,E_2$ are $(-2)$-curves and 
that  the image of the natural map
\[
\phi: H^0(D-E_1)\oplus H^0(D-E_2)\to H^0(D)
\]
has codimension two. Then $R(X)$ has no generator in degree $[D]$.
\end{proposition}

The following is a result by Ottem  \cite[Proposition 2.2]{JCO}.

\begin{proposition}\label{ottem}
Let $X$ be a smooth projective K3 surface. Let $A$ and $B$ be nef divisors on $X$ such that $|B|$ is base point free. Then the multiplication map
\[
H^0(X,A)\otimes H^0(X,B)\to H^0(X,A+B)
\]
is surjective if $h^1(X,A-B)=h^1(X,A)=0$ and $h^2(X,A-2B)=0$.
 \end{proposition}
 
 On the other hand, the following results show that under suitable conditions  
$R(X)$ must have a generator in a certain degree.
 
  \begin{lemma}\label{l2}
 Let $X$ be a K3 surface and let $D=E_1+E_2+E_3$ be a base point free divisor, 
 where $E_1,E_2,E_3$  are $(-2)$-curves such that $h^1(E_i+E_j)=0$ 
 for all distinct $i,j$. Then the natural map
 \[
 \psi:\bigoplus_{i=1,2,3} H^0(D-E_i)\to H^0(D)
 \]
 is not surjective. Moreover, if $E_1\cap E_2\cap E_3=\emptyset$, then the image of $\psi$ 
 has codimension one.
 \end{lemma}
 
% \begin{proof}
% If $E_1\cap E_2\cap E_3$ is not empy, then $\psi$ 
% is clearly not surjective, since $D$ is base point free.
%On the other hand, if $E_1\cap E_2\cap E_3$ is empty,
%then we can consider the associated Koszul exact sequence 
%of sheaves in Theorem \ref{koszul3}, which gives rise 
%to the two short exact sequences \eqref{k1} and \eqref{k2}.
%The first sequence, using the fact that $h^0(E_i)=1, h^1(E_i)=h^2(E_i)=0$  
%for $i=1,2,3$ and $h^1(\Osh_X)=0$, $h^2(\Osh_X)=1$, gives 
%$h^0({\rm Im}(d_2))=2$ and $h^1({\rm Im}(d_2))=1$.
%Using this and the fact that $h^1(E_i+E_j)=0$ in the second sequence, 
%we find that the image of $\psi$ has codimension one in $H^0(D)$.
% \end{proof}

 \begin{proposition}\label{minimal}
Let $G=\{w_0,\dots,w_r\}\subset \Cl(X)$ containing the degrees of a homogeneous 
generating set of $R(X)$ and suppose that the linear system associated to $w_0\in G$  is base point free.
Then $R(X)$ has a generator in degree $w_0\in G$ if one of the following holds:
\begin{enumerate}[$i.$]
\item any linear combination $w_0=\sum_{i=1}^r a_iw_i$ with $a_i\in \zz, a_i\geq 0$ 
contains the class of a $(-2)$-curve in its support;
\item any linear combination $w_0=\sum_{i=1}^r a_iw_i$ with $a_i\in \zz, a_i\geq 0$ 
contains in its support one of the classes of  two $(-2)$-curves 
$E_1,E_2$ with $E_1\cdot E_2>0$;
\item  $w_0=w_1+w_2+w_3$, where $w_i$ are the classes 
of three $(-2)$-curves with $h^1(w_i+w_j)=0$ for all distinct $i,j$, 
and any linear combination $w_0=\sum_{i=1}^r a_iw_i$ with 
$a_i\in \zz, a_i\geq 0$ contains one among $w_1,w_2,w_3$ in its support. 
\end{enumerate}
\end{proposition}

%\begin{proof}
%Let $L$ be the linear system associated to $w_0$, let $S$ be the subspace 
%of $H^0(w_0)$ whose elements are polynomials in elements of $H^0(w_i)$, $i=1,\dots,r$, 
%and let $L_S$ be the corresponding subspace of $L$. 
%The hypothesis in $i.$ says that any divisor in $L_S$ contains $E$ in its support. 
%Similarly, the hypothesis in $ii.$ says that any divisor in $L_S$ contains $E_1\cap E_2$ in its support.
%Since $L$ is base point free, this implies that $L\not=L_S$, so that $R(X)$ has a generator in degree $w_0$. 
%Finally, if the hypothesis in $iii.$ holds,  $S$ is equal to the image of $\psi$ in Lemma \ref{l2}. 
%On the other hand, since $w_0$ satisfies the hypothesis of Lemma \ref{l2},
% $\psi$ is not surjective, i.e. $S\not=H^0(w_0)$. Thus $R(X)$ has a generator in degree $w_0$.
%\end{proof}

The previous results have been implemented in several Magma \cite{B.C.P} programs 
included as ancillary files in the arXiv version of the paper \href{https://arxiv.org/abs/1909.01267}{https://arxiv.org/abs/1909.01267}.

\section{K3 surfaces of Picard number four}
%\subsection{Cox rings}
In this section we study the geometry of each family of Mori dream K3 surfaces of Picard number four.
As a first step, we compute their effective cone and their nef cone.

\begin{theorem}\label{main-eff}
For each of the $14$ families of Mori dream K3 surfaces $X$ of Picard number four, Tables \ref{Table1}, \ref{Table2} and \ref{Table3} describe 
\begin{enumerate}
\item the set $E(X)$ of extremal rays of the effective cone and its Hilbert basis ${\rm BEff}(X)$, 
\item the set of extremal rays $N(X)$ of the nef cone  and its Hilbert basis ${\rm BNef}(X)$. 
\end{enumerate}
Moreover, Table \ref{Intmat} gives the intersection matrix of the $(-2)$-curves for each family.
\end{theorem}

\begin{proof}
A set of fundamental roots for the Picard lattice is obtained by means of the algorithm described in section \ref{eff-nef}, implemented in the Magma library {\rm Find-2} in \cite{A.C.L}. The nef cone is thus obtained as the dual of the effective cone with respect to the intersection form of Cl(X). 
\end{proof}

%\begin{theorem}\label{main-cox}
%Let $X$ be a  Mori dream K3 surface of Picard number four 
%such that $\NS(X)$ is not isometric to $V_{14}$.
%The degrees of a set of generators of the Cox ring $R(X)$ are given in
%Table \ref{TableGen rank4}. All degrees in the Table are necessary to generate 
%$R(X)$, except eventually for those marked with a star.
%\end{theorem}
We now prove the result about the Cox ring of such surfaces stated in the Introduction.\\

\noindent {\em Proof of Theorem \ref{main-cox}.}
%\begin{proof}
 By the remarks in section \ref{cox} the Cox ring has a generator in each degree $w\in {\rm BEff}(X)$.
Observe that if $D$ is not nef, then there exists a $(-2)$-curve $C$ such that $D\cdot C<0$, 
so that $C$ is  contained in the base locus of $|D|$ and the multiplication map $H^0(X,D-C)\rightarrow H^0(X,D)$ 
by a non-zero element of $H^0(X,C)$ is surjective. Thus we can assume $[D]$ to be nef.

%Moreover, there are two cases where the K3 surface has two smooth elliptic curves $F,F'$ with $F\cdot F'=2$, 
 %the families of K3 surfaces with Picard lattice isometric to $V_5= U(2)\oplus 2A_1$ and $V_9=U(2)\oplus A_2$. 
%The cases where $\NS(X)\cong V_8$ or $V_9$ correspond to families $\mathcal F_8$ and $\mathcal F_9$ respectively (see Proposition \ref{gen v8} and Proposition \ref{gen v9}).
%For the following arguments we exclude the above cases, which will be described in the next section.
  
By Theorem \ref{gen} it is enough to consider those nef degrees which are sums of at most 
three elements of the Hilbert basis of the nef cone or of the form $[2(F+F')]$, 
where $F,F'$ are smooth elliptic curves with $F\cdot F'=2$. The last type of degrees only exist for 
the families of K3 surfaces with N\'eron-Severi lattice isometric to either $V_5= U(2)\oplus 2A_1$ or $V_9=U(2)\oplus A_2$. 
In case $\NS(X)\cong V_5$ the Cox ring has been computed in \cite[Proposition 6.7, i)]{A.H.L}.

This result allows to form a finite list of possible nef degrees, which is then analyzed using the 
techniques in section \ref{cox}, which have been implemented in several Magma programs \cite[\S 3.3]{A.C.L}.
 In particular Theorem \ref{teokoszul}, Proposition \ref{va} and Proposition \ref{ottem} allow to discard certain degrees, while Lemma \ref{l2} and Proposition \ref{minimal} are used to prove that certain degrees are necessary. 
 
We recall that, by Theorem \ref{bl}, the linear system of any effective nef divisor is base point free unless there exists 
a smooth elliptic curve $F$ and a $(-2)$-curve $E$ such that $E\cdot F=1$.
This happens only for the families of K3 surfaces with 
N\'eron-Severi lattice isometric to either $V_4=U\oplus 2A_1$ or $V_8=U\oplus A_2$. 
In case $\NS(X)\cong V_4$ the Cox ring has been computed in \cite[Proposition 6.7, ii)]{A.H.L}.
In case  $\NS(X)\cong V_8$ we apply the same techniques, but we exclude from the computations the only element in the Hilbert basis of the nef cone  which is not base point free, which is the class $\BNef[4]$.

In case $\NS(X)\cong V_{14}$ we are unable to compute a minimal generating set of $R(X)$ for computational reasons (the Hilbert basis of the nef cone contains $111$ elements). 
To obtain the set of degrees in Table  \ref{TableGen rank4} we only apply  item i) of Theorem \ref{teokoszul} and Proposition \ref{minimal}. \qed\\

%By Theorem \ref{classification} there are $14$ families of K3 
%surfaces with Picard number four.
%E. vinberg in \cite{EV1} found and studied all 2-reflective hyperbolic lattices of rank four.
 
%In this section we will determine the degrees of a generating set of $R(X)$ for each such family.

%\subsection{Effective cones and projective models}\label{effective cone}
%As observed in section \ref{cox}  the Cox ring has a 
%generator in degree $[D]$ for each $[D]$ in the Hilbert basis of ${\rm Eff}(X)$.
%Moreover, by Theorem \ref{gen}, the Hilbert basis of the nef cone ${\rm Nef}(X)$ 
%also has a key role in the computation of $R(X)$.
%We computed such bases by means of a computer program written in Magma, 
%see section \ref{magma}.
%
%
%\begin{proposition}\label{eff-nef}
%Tables $1,2$ and $3$  describe the extremal rays and the Hilbert bases of ${\rm Eff}(X)$ and ${\rm Nef}(X)$ 
%for each of the $14$ families of Mori dream K3 surfaces of Picard number four.
%\end{proposition}
%
%  
%In the tables and in the proof of the following theorem we will adopt this notation:
%
%\begin{itemize}
%\item $\Eff(X)$ is the effective cone of $X$, $E(X)$ the list of its extremal rays (i.e. the classes of  $(-2)$-curves) and $\BEff(X)$ its Hilbert basis,
%\item $\Nef(X)$ is the nef cone of $X$, $N(X)$ the list of extremal rays of $\Nef(X)$ and $\BNef(X)$ its Hilbert basis.
%\end{itemize}
%\subsection{Projective models}
 We now  analyze each family of Mori dream K3 surfaces
of Picard number four using the results in Theorem \ref{main-eff} and Theorem \ref{main-cox}.
We will denote by $\mathcal F_i$ the family 
of K3 surfaces whose N\'eron-Severi lattice is isometric to $V_i$
for $i=1,\dots,14$. 

\subsection*{The family $\mathcal F_1$}

Let $X$ be a K3 surface with $\NS(X)\cong V_1=(8)\oplus 3A_1$.
By Theorem \ref{main-eff} $X$ contains  $12$  $(-2)$-curves 
whose intersection matrix is given in Table \ref{Intmat}.
Moreover,  the Hilbert basis of the nef cone of $X$ contains $51$ classes, 
and exactly six of them are classes of elliptic fibrations: $\BNef[1]$, $\BNef[6]$, $\BNef[11]$, $\BNef[27]$, $\BNef[29]$ and $\BNef[35]$. 
Each elliptic fibration has two fibers of type $\tilde A_1$ and has no sections.

%and by Proposition \ref{hyp} is hyperelliptic class.

\begin{proposition}\label{gen v1}
Let $X$ be a K3 surface with $\NS(X)\cong V_1=(8)\oplus 3A_1$. 
Then 
\begin{enumerate}
\item $X$ is birational to a complete intersection $\overline X$ of three quadrics in $\pp^5$ of the form
\[
Q_1(x_1,x_2,x_3,x_4, x_5)=Q_2(x_0,x_2,x_3, x_4,x_5)=Q_3(x_0,x_1,x_3,x_4,x_5)=0,
\]
where $Q_1,Q_2,Q_3$ are homogeneous of degree two and do not contain the monomials 
$x_0^2, x_1^2, x_3^2$;
\item the projection of $\overline X$ from the plane $x_0=x_1=x_2=0$  induces a double cover  
$\pi:X\to \pp^2$ branched along a smooth plane sextic with six $6$-tangent conics $C_1,\dots,C_6$;
\item the surface $X$ has twelve $(-2)$-curves: the curves $R_{ij}$, $i=1,\dots,6$, $j=1,2$, such that 
$\pi(R_{i1})=\pi(R_{i2})=C_i$;
 \item a minimal generating set of the Cox ring of $X$ is
 $s_1, \dots,   s_{15}$, 
 where $s_1,\dots,s_{12}$ are defining sections of the $(-2)$-curves 
 and $s_{13}, s_{14}, s_{15}$ is a basis of $H^0(\pi^*\oo_{\pp^2}(1))$.
%where $s_i$ be a generator of $H^0(f_i)$ for all i, and a basis $x_0, x_1, x_2\in H^0 (h)$. 
%whose degrees are given by the columns of the following matrix
%\[
%\left(
%\begin{array}{ccccccccccccccc}
%1& 1 & 0 & 2 & 0 & 2 & 1 & 1 & 1 & 1 & 2 & 0 & 1&1&1\\
%-1& -1 & 0 & -2 & 1 & -2 & 0 & -2 & 0 & -2 & -3 & 0 & -1&-1&-1\\
%0& 2 & -1 & 2 & 0 & 3 & 1 & 0 & 2 & 1 & 2 & 0 & 1&1&1\\
%-2& 0 & 0 & -3 & 0 & -2 & -2 & -1 & -1 & 0 & -2 & 1 & -1&-1&-1
%\end{array}
%\right).
%\]
\end{enumerate}
\end{proposition}

\begin{proof}
The class $h=\BNef[8]$ has square $8$, it is base point free by Corollary \ref{section} 
and is non-hyperelliptic (the intersection numbers with the fibers of the $6$ elliptic fibrations are $\geq 8$).
Thus it defines a morphism $\varphi_h:X\to \pp^5$ whose image is a degree eight surface $\overline X$ in $\pp^5$ with three nodes $p_1, p_2,p_3$ at the images of the three disjoint $(-2)$-curves of classes $f_3, f_5, f_{12}$.
By \cite[Theorem 7.2]{SD} the surface $\overline X$ is a complete intersection of three 
quadrics, since the  intersection numbers of $h$ with the fibers 
of the elliptic fibrations of $X$ are $>3$.

If the points $p_{1},p_{2},p_{3}$ were not in general
position, then $\overline X$ would contain the line $L$ through them. 
The strict transform of $L$ and the exceptional divisors over the nodes would 
thus generate a rank four lattice isometric to $D_4$. However,  looking at the intersection matrix 
of the $(-2)$-curves of $X$ (see Table \ref{Intmat}) one can see that there are no four of them with intersection matrix isometric to $D_4$, 
giving a contradiction.
% On the desingularization,
%there are the strict transform $C$ of that line and the $\cu$-curves
%$C_{1},C_{2},C_{3}$ over the singularities. These $\cu$-curves are
%such that $CC_{j}=1$, $C_{j}C_{k}=0$ for $j\neq k$. Such a surface
%is not in the family of K3 surfaces with $\NS X)\simeq[8]\oplus\mathbf{A}_{1}^{\oplus3}$.
%\\
We can therefore suppose that the points $p_{1},p_{2},p_{3}$ are
in general position, and up to projectivities we can assume 
that: 
\[
p_1=(1,0,0,0,0,0),\ p_2=(0,0,1,0,0,0),\  p_3=(0,0,1,0,0,0).
\]
Let $\mathcal N$ be the net of quadrics containing $\overline X$.
It can be proved that there is a unique quadric $Q_i$ in $\mathcal N$ 
having a node at $p_i$ for $i=1,2,3$.
%\michela{should we add more explanation here?}
%Let $\ell_{k}$ be the equation of the tangent space of the quadric $Q_{k}$ at
%$p_{1}$. It can be easily proved that the complete intersection surface $\overline X$
%has a node at $p_{1}$ if and only if the linear system generated
%by $\ell_{1},\ell_{2},\ell_{3}$ is a pencil. 
%Thus there is a unique quadric in the net $\mathcal{N}$ generated by 
%$Q_1,Q_2,Q_3$ which is singular at $p_{1}$. 
Thus that there are three possible cases, up to renumbering the three quadrics:\\

\noindent a) $Q_1, Q_2, Q_3$ are three distinct quadrics which generate $\mathcal N$;\\
b)  $Q_{1}=Q_{2}$, $Q_1$ is singular along the line $L$ between $p_{1}, p_{2}$ and
$Q_3$ is singular along $p_3$  (the line $L$ is not contained in $\overline X$ 
since otherwise $\overline X$ would be singular along $L$);\\
c) $Q_1=Q_2=Q_3$ and $Q_1$ is singular along the plane containing the points $p_1,p_2,p_3$.\\

Case c) is not possible, since $Q_1$ would be singular along the plane
$\Pi$ containing $p_1,p_2,p_3$ and $\Pi\cap \overline X$ would consist of $p_1,p_2,p_3$ 
and a fourth singular point.  Moreover, case b) can be proved to be a degeneration of case a) 
(in fact, complete intersections of three quadrics in $\pp^5$ of type $b)$ can be proved to 
have $14$ moduli).
%$(20-3)+(20-3-5)+(20-3-2\cdot 5)-2-20=14$ moduli).
%\michela{should we add more explanation here?}

Thus we can assume to be in case $a)$ and the defining equations of the quadrics
$Q_1, Q_2, Q_3$ (denoted in the same way) are as in item i) of the  statement.

%Let $f_1,\dots, f_{12}$ be the classes of the $(-2)$-curves ordered as in Table \ref{Table1} 
Consider the class $\BNef[4]=h-f_3-f_5-f_{12}$.
Then $h^2=2$ and $h\cdot f_i=2$ for all $i$,  thus $h$ is ample and 
the associated linear system is base point free by Corollary \ref{section}.
Thus $h$ defines a morphism $\pi:X\to \pp^2$ branched along a smooth plane sextic $B$.
Since 
\[
2h=f_7+f_{10}=f_4+f_{12}=f_8+f_9=f_3+f_6=f_5+f_{11}=f_{1}+f_{2},
\]
the image by $\pi$ of the twelve $(-2)$-curves of $X$ are six smooth conics $C_1,\dots,C_6\subseteq \pp^2$ 
such that $\pi^{-1}(C_i)$ is the union of two smooth rational curves $R_{ij}$ with $j=1,2$, for any $i=1,\dots,6$. 
This implies that the conics are tangent to $B$ at $6$ points.

By Theorem \ref{main-cox} the Cox ring $R(X)$ is generated in the following degrees: 
\[
f_1,\dots, f_{12},\ \BNef[4]. 
\]
Observe that $\BNef[4]$ is also an element of the Hilbert basis of the effective cone of $X$ (see Table \ref{Table1}).
Clearly any minimal generating set of $R(X)$ must contain the sections $s_1,\dots,s_{12}$ defining 
the $(-2)$-curves of $X$ and a basis $s_{13}, s_{14}, s_{15}$ of $H^0(\BNef[4])$, thus $\{s_1,\dots,s_{15}\}$ is a minimal generating set of $R(X)$.
\end{proof}

\begin{corollary}\label{uni1}
The moduli space of K3 surfaces with $\NS(X)\simeq (8)\oplus 3{A}_{1}$
is unirational. 
\end{corollary}

\begin{proof} 
Let $\mathcal Q$ be the projective linear system of quadrics through the points 
$p_1,p_2,p_3$ (of dimension $17$) and let $Gr$ be the Grassmannian 
of nets in $\mathcal Q$. 
Let $\mathcal A\subset \mathcal Gr$ be the space of nets of quadrics of type $a)$. 
Observe that, by the proof of Proposition \ref{gen v1}, giving a net in $\mathcal A$ is equivalent 
to give three quadrics $Q_1, Q_2, Q_3\in \mathcal Q$, such that $Q_i$ has a node at $p_i$ for $i=1,2,3$. 
Observe that having a node at $p_i$ imposes $5$ conditions on the coefficients of a quadric in $\mathcal Q$.
Thus $\mathcal A$ is a rational variety of projective dimension $3(\dim(\mathcal Q)-5)= 36$. 
By item i) of Proposition \ref{gen v1} the moduli space of K3 surfaces with $\NS(X)\simeq(8)\oplus 3{A}_{1}$ is birational to the quotient of $\mathcal A$ by the subgroup $G$ of ${\rm PGL}_6(\cc)$ fixing $p_1,p_2,p_3$,
whose elements are matrices of the form 
\[
\left(\begin{array}{cccccc}
* & 0 & 0 & * & * & *\\
0 & * & 0 & \vdots & \vdots & \vdots\\
0 & 0 & * & \vdots & \vdots & \vdots\\
0 & 0 & 0 & \vdots & \vdots & \vdots\\
0 & 0 & 0 & \vdots & \vdots & \vdots\\
0 & 0 & 0 & * & * & *
\end{array}\right)
\]
and thus has dimension $20$.
\end{proof}

\subsection*{The family $\mathcal F_2$}

Let $X$ be a K3 surface with $\NS(X)\cong V_2=(4) \oplus (-4) \oplus A_2$.
By Theorem \ref{main-eff} $X$ contains six $(-2)$-curves whose intersection 
matrix is given in Table \ref{Intmat}.
The Hilbert basis of the nef cone contains $35$ classes, 
with four classes of elliptic fibrations: $\BNef[5]$, $\BNef[11]$, $\BNef[23]$ and $\BNef[29]$. 
Each elliptic fibration has one singular fiber of type $\tilde A_2$ and has no sections.

\begin{proposition}\label{gen v2}
Let $X$ be a K3 surface with $\NS(X)\cong  V_2=(4) \oplus (-4) \oplus A_2$. 
Then 
\begin{enumerate}
\item $X$ is birational to a quartic surface $\overline X\subset \pp^3$ with a node and with two hyperplane 
sections passing through the node which are the union of two conics;
\item the projection of $\overline X$ from the node induces a double cover $\pi:X\to \pp^2$ branched along a smooth plane sextic with two $3$-tangent lines $L_1,L_2$ 
and one $6$-tangent conic $C$;
\item the surface has six $(-2)$-curves: the curves $R_{ij}$, $i,j=1,2$ such that 
$\pi(R_{i1})=\pi(R_{i2})=L_i$ and the curves $S_1,S_2$ such that 
$\pi(S_1)=\pi(S_2)=C$;
\item the Cox ring of $X$ is generated in the degrees given in Table \ref{TableGen rank4} and 
has at least $23$ generators in such degrees:
$s_1,\dots, s_6$ defining the $(-2)$-curves, $s_7,\dots, s_{10}$ defining each a smooth fiber 
of the four elliptic fibrations, $s_{11}\in H^0(h-f_3)$ independent from the elements defining 
$\pi^{-1}(L_1), \pi^{-1}(L_2)$, $s_{12},\dots,s_{23}$ 
whose degrees are elements of the Hilbert basis of the nef cone 
with self-intersections $4$ (for $i=12,\dots,15$), $10$ (for $i=16,\dots,19$) and $12$ (for $i=20,\dots,23$).
%The Cox ring of a very general $X$ has at least eleven generators.
%where $s_i$ be a generator of $H^0(f_i)$ for all i, and a basis $x_0, x_1, x_2\in H^0 (h)$. 
\end{enumerate}
\end{proposition}

\begin{proof}
Let $f_i$ the classes of the $(-2)$-curves for $i=1,\dots, 6$, and let $h=\BNef[7]$.
Then $h^2=4$, $h\cdot f_i=2$ for $i=2,4,5,6$, $h\cdot f_1=8$ and $h\cdot f_3=0$. 
By Corollary \ref{section} the linear system associated to $h$ is base point free.
Moreover, the intersections of $h$ with all classes of elliptic fibrations are $\geq 4$.
Thus $h$ defines a birational morphism $\varphi:X\to \pp^3$ whose image is a quartic 
surface $\overline X\subset \pp^3$ with a node. Since $f_2+f_5=h-f_3$, then the $(-2)$-curves 
of classes $f_2, f_5$ are the proper transforms of two conics lying in a hyperplane section of $\overline X$ 
passing through the node. The same holds for $f_4, f_6$.

Now consider the class $k=\BNef[15]=h-f_3$. 
Then $k^2=2$, $k\cdot f_1=k\cdot f_3=2$ and $k\cdot f_i=1$ for $i\neq 1,3$. 
Thus $k$ is ample and the associated linear system $|k|$ is base point free by Corollary \ref{section}. 
Thus it defines a  double cover $\pi:X\to \pp^2$ branched along a smooth plane sextic $B$.
Since $k=f_2+f_5=f_4+f_6$ and $2k=f_1+f_3$, the image by $\pi$ of the six $(-2)$-curves of $X$ 
are a smooth conic $C \subseteq \pp^2$ and two smooth lines $L_1,L_2\subseteq \pp^2$.
%
% such that $\pi^{-1}(C)$ is the union of two smooth rational curves and $\pi^{-1}(L_i)$ is the union of two smooth rational curves , i. e. the curves $R_{ij}$, $i,j=1,2$ such that 
%$\pi(R_{i1})=\pi(R_{i2})=L_i$ and the curves $S_1,\ S_2$ such that 
%$\pi(S_1)=\pi(S_2)=C$.
The last statement follows from Theorem \ref{main-cox}. 
Observe that we clearly need one generator for each $(-2)$-class,
one generator for each elliptic fibration (since one of its fibers is reducible) and
one generator in degree $k$.
In the remaining degrees, we know that there is at least one generator of $R(X)$
by Proposition \ref{minimal}.
%By Theorem \ref{main rank4} the following are the degrees of a minimal generating set of the Cox ring $R(X)$: 
%\[
%f_1,\dots, f_{6},\ h,\ e_1,e_2,e_3,e_4,\ h_1, \dots, h_{12},
%\]
%where $e_1,e_2,e_3,e_4$  are four classes of elliptic fibrations 
%and the Hilbert basis of the nef cone contains the classes $h_i$ for $i=1,\dots,12$.
%Clearly any minimal generating set of $R(X)$ must contain the sections $s_1, \dots, s_{6}$ defining 
%the $(-2)$-curves of $X$, a basis $s_2s_5,\ s_4s_6, \ x_0$ of $H^0(h)$ and one generator for each elliptic fibration.
%By Proposition \ref{hyp}, the classes $h_i$ for $i=1,\dots,12$ are non-hyperelliptic and if $v:=\BNef[i]$ we have to
%\begin{enumerate}
%\item $v^2=4$ for $i=7,\ 14,\ 20,\ 25$
%\item $v^2=10$ for $i=3,\ 4,\ 32,\ 33$
%\item $v^2=12$ for $i=1,\ 2,\ 34,\ 35$
%\end{enumerate}
%and we don't know how many generators there are without knowing generators relations of Cox ring.
%Therefore, the Cox ring of a very general $X$ has at least eleven generators.
\end{proof}

%\michela{projective model still missing}

%\begin{corollary}
%The moduli space of K3 surfaces with $\NS(X)\cong V_2$ is unirational.
%\end{corollary}
%
%\begin{proof}
%
%
%\end{proof}

\subsection*{The family $\mathcal F_3$}

Let $X$ be a K3 surface with $\Cl(X)\cong V_3=(4) \oplus A_3$.
By Theorem \ref{main-eff} $X$ contains five  $(-2)$-curves  
whose intersection matrix is given in Table \ref{Intmat} and described in Figure \ref{fig1}.

\begin{figure}[H]
\begin{center}
\begin{tikzpicture}[scale=1]

\draw (0,0) -- (0.866,0.5);
\draw (0,0) -- (0.866,-0.5);
\draw (0,0) -- (-0.866,0.5);
\draw (0,0) -- (-0.866,-0.5);

\draw [very thick] (0.866,0.5) -- (0.866,-0.5);
\draw [very thick] (-0.866,0.5) -- (-0.866,-0.5);

\draw (0,0) node {$\bullet$};
\draw (0.866,0.5) node {$\bullet$};
\draw (0.866,-0.5) node {$\bullet$};
\draw (-0.866,0.5) node {$\bullet$};
\draw (-0.866,-0.5) node {$\bullet$};

\draw (0,0) node [above]{$E_{5}$};
\draw (0.866,0.5) node [right]{$E_{2}$};
\draw (0.866,-0.5) node [right]{$E_{3}$};
\draw (-0.866,0.5) node [left]{$E_{1}$};
\draw (-0.866,-0.5) node [left]{$E_{4}$};

\draw (-1.1,0.2) node [below]{\small $2$};
\draw (1.1,0.2) node [below]{\small $2$};
\end{tikzpicture}
\end{center}
\caption{Intersection graph of $(-2)$-curves for $\mathcal F_3$}\label{fig1}
\end{figure}
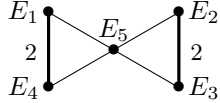

The Hilbert basis of the nef cone of $X$ contains $10$ classes.
 Exactly one of them, $\BNef[3]$,  is the class of an elliptic fibration with
two reducible fibers of type $\tilde A_1$ and without sections.
%Let $f_i$ the class of the $(-2)$-curve for $i=1,\ 2,\dots, 5$, and $h=\BNef[2]$ a unique class such that $h^2=2$, 
%with $h\cdot f_5=0$ and $h\cdot f_i=1$ for $i\neq 5$, $h^0(h)=3$, $h^1(h)=0$ and by Proposition \ref{hyp} is hyperelliptic class.

\begin{proposition}\label{gen v3}
Let $X$ be a K3 surface with $\NS(X)\cong  V_3=(4) \oplus A_3$. Then
\begin{enumerate}
\item $X$ is birational to a double cover 
$\pi:\overline X\to \pp^2$ branched along a plane sextic $B$ with a double point $p$ with two 
bitangent lines $L_1,L_2$ passing through $p$;
\item the surface $X$ has five $(-2)$-curves: 
the exceptional divisor $E$ over the singular point $\pi^{-1}(p)$ 
and four curves $R_{ij}$, with $i,j=1,\ 2$, such that $\pi(R_{ij})=L_{i}$;
\item 
the Cox ring of   $X$ is generated in the degrees given in Table \ref{TableGen rank4} and 
has at least ten generators in such degrees: 
$s_1,\dots, s_5$ defining the $(-2)$-curves, $s_6$ defining the pull-back of a line 
in $\pp^2$ not passing through $p$
and $s_7,\dots,s_{10}$ whose degrees are four distinct classes in the Hilbert basis of the nef cone 
with self-intersection $4$.
%where $s_i$ be a generator of $H^0(f_i)$ for all i, and a basis $x_0, x_1, x_2\in H^0 (h)$. 
\end{enumerate}
\end{proposition}

\begin{proof}
Let $f_1,\dots, f_5$ be the classes of the $(-2)$-curves 
and $h=\BNef[2]$. Then  $h^2=2$,
 $h\cdot f_5=0$ and $h\cdot f_i=1$ for $i\neq 5$. Thus $h$ is nef.
 %$h^0(h)=3$, $h^1(h)=0$ and by Proposition \ref{hyp} is hyperelliptic class.
 By Corollary \ref{section}, the linear system associated to $h$ is base point free 
 and thus defines a morphism $\psi:X\to \pp^2$ which contracts the $(-2)$-curve 
 with class $f_5$ to a point $p\in \pp^2$ and is branched along a plane sextic $B$
 with one node at $p$.
 Since $h=f_1+f_4+f_5=f_2+f_3+f_5$ the image by $\psi$ of the four $(-2)$-curves  
 of classes $f_1,\dots,f_4$ are two lines $L_1,L_2$ passing through $p$ 
 and tangent to $B$ in two more points.

 By Theorem \ref{main-cox}, a minimal set of generators of the Cox ring $R(X)$ has 
 the following degrees: 
\[
 f_1,\dots, f_{5},\ h,\ h_1,h_2,h_3,h_4, 
\]
where $h_i=\BNef[i]$ with $i=1,4,7,9$.  
Clearly any minimal generating set of $R(X)$ must contain the sections $s_1,\dots,s_{5}$ defining 
the $(-2)$-curves of $X$ and an element $s_6$ such that  $s_1s_4s_5, s_2s_3s_5, s_6$ is a basis of $H^0(h)$.
In the remaining degrees $h_1,\dots,h_4$ we known that there is at least one generator by Proposition \ref{minimal}.
%By Proposition \ref{hyp}, the classes $h_i$ are non-hyperelliptic, ${h_i}^2=4$ and $H^0(h_i)=4$, for $i=1,\dots,\ 4$
%\begin{align*} 
%h_1&=f_1+f_2+2f_4+2f_5=f_2+f_4+f_5+h=2f_2+f_3+f_4+2f_5\\
%h_2&=2f_1+f_2+f_4+2f_5=f_1+f_2+f_5+h=f_1+2f_2+f_3+2f_5\\
%h_3&=f_3+f_4+f_5+h=f_2+2f_3+f_4+2f_5=f_1+f_3+2f_4+2f_5\\
%h_4&=f_1+f_3+f_5+h=f_1+f_2+2f_3+2f_5=2f_1+f_3+f_4+2f_5
%\end{align*}
%Therefore, the Cox ring of a very general $X$ has at least ten generators.
\end{proof}

\begin{proposition}\label{uni3}
The moduli space of K3 surfaces with $\NS(X)\simeq (4)\oplus A_{3}$
is rational. 
\end{proposition}

\begin{proof}
Let $q$ be the nodal point of the branch locus $B$ and let $p_{1},p_{2}$
(respectively $r_{1},r_{2})$ the other intersection points of the tangent
line $L_{1}$ (respectively $L_{2}$) to the sextic. Using automorphisms of $\pp^2$,
we can impose that 
\[
q=(0,1,0),\,p_{1}=(1,0,0),\,p_{2}=(-1,1,0),\,r_{1}=(0,0,1),\,r_{2}=(0,1,-1),
\]
and there are no more symmetries. One computes that the projective
linear system of smooth plane sextics with a node at $q$ and tangent to the lines
$L_{1},L_{2}$ at $p_{1},p_{2},r_{1},r_{2}$ is an open set in the
affine space $\mathbb{A}^{16}$, therefore the moduli space of K3 surfaces with
$\NS(X)\simeq(4)\oplus A_{3}$ is rational. 
\end{proof}

\begin{remark}
The surface $X$ admits four different morphisms $X\to \pp^3$ mapping onto a quartic surface 
with a singular point of type $A_3$, defined by the classes $\BNef[i]$, $ i=2,4,7,9$.
%One could have also searched for a birational model in $\pp^{3}$ with a ${A}_{3}$-singularity
%; there are $4$ such models, by example the divisor $D_{4}=2A_{1}+2A_{2}+A_{3}+A_{4}$
%give such a birational map $X\to\pp^{3}$. 
\end{remark}

\subsection*{The family $\mathcal F_4$}

Let $X$ be a K3 surface with $\NS(X)\cong V_4=U \oplus 2A_1$.
By Theorem \ref{main-eff}, $X$ contains five  $(-2)$-curves 
whose intersection matrix is given Table \ref{Intmat} and described in Figure \ref{fig2}.
The Hilbert basis of the nef cone of $X$ contains five classes, 
one of them is the class $\BNef[5]$ of an elliptic fibration $\BNef[5]$ with a section 
and two reducible fibers of type $\tilde A_1$. 

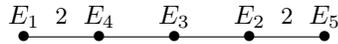
\begin{figure}[H]\label{fig2}
\begin{center}
\begin{tikzpicture}[scale=1]

\draw (0,0) -- (4,0);
%\draw (4,1) -- (4,-1);

%\draw [very thick] (0,0) -- (1,0);
%\draw [very thick] (3,0) -- (4,0);

\draw (0,0) node {$\bullet$};
\draw (1,0) node {$\bullet$};
\draw (2,0) node {$\bullet$};
\draw (3,0) node {$\bullet$};
\draw (4,0) node {$\bullet$};

\draw (0,0) node [above]{$E_{1}$};
\draw (1,0) node [above]{$E_{4}$};
\draw (2,0) node [above]{$E_{3}$};
\draw (3,0) node [above]{$E_{2}$};
\draw (4,0) node [above]{$E_{5}$};

\draw (0.5,0.5) node [below]{\small $2$};
\draw (3.5,0.5) node [below]{\small $2$};

\end{tikzpicture}
\end{center}
\caption{Intersection graph of $(-2)$-curves for $\mathcal F_4$} 
\end{figure}

%Observe that the $(-2)$-curves are the section and the components of the reducible fibres.

%Let $f_i$ the class of the $(-2)$-curve for $i=1,\ 2,\dots, 5$ and let the class $h=\BNef[4]$ hyperelliptic by Proposition \ref{hyp} and a unique class such that  $h^2=2$, 
%with $h\cdot f_i=0$ for $i=1,\ 3,\ 5$ and $h\cdot f_i=1$ for $i=2,\ 4$, $h^0(h)=3$, $h^1(h)=0$ and 
%$h=f_1+f_2+f_3+f_4+f_5=2f_2+f_3+2f_5=2f_1+f_3+2f_4$. 

%The surface $X$ has an elliptic fibration $\pi:X\to \pp^1$ 
% with a section and two reducible fibers of type  $\tilde A_1$.  
 %The surface has five $(-2)$-curves: the components of the reducible fibers  of $\pi$ and the section.

 In this case, the Cox ring has been computed in \cite[Proposition 6.7, ii)]{A.H.L},
 where $X$ is described as a double cover of the Hirzebruch surface $\ff_4$ blown-up  at two general points $p,q$.
 We will use the following notation for curves in $\ff_4$: $F_1, F_2$ are the two fibers of the projection $\ff_4\to \pp^1$ 
 passing through $p$ and $q$, $S_1$ is the curve in $\ff_4$ with $S_1^2=-4$ and
 $S_2$ is a curve passing through $p$ and $q$ with $S_2^2=4$ and $S_1\cdot S_2=0$. We will call $S_1$ also its pull-back in the blow-up of $\ff_4$ at $p,q$.
 
 We recall these properties in the following proposition:
%What was said previously is described in the following proposition:
 
% In  \cite[Proposition 6.7, ii)]{A.H.L} $X$ is also described as the double cover 
%of a smooth quadric surface blown-up at two general points.
%\michela{agregar modelo en articulo de Rolleau?}

\begin{proposition}\label{gen v4}
Let $X$ be a K3 surface with $\NS(X)\cong  V_4=U \oplus 2A_1$. Then
\begin{enumerate}
\item there is a double cover  $\pi:X\to Y$, where $Y$ is a blow-up of $\ff_4$ at two general points $p,q$, 
branched along $S_1$ and a smooth curve $B$ of genus $8$;
\item the surface has five $(-2)$-curves, which are the preimages by $\pi$ of: $S_1$,
the proper transforms of  $F_1,F_2$ and the two exceptional divisors over $p$ and $q$;
 \item a minimal generating set of the Cox ring of  $X$  is $s_1,\dots,s_7$, where $s_1, \dots, s_{5}$ define the $(-2)$-curves, $s_6$ defines the preimage of the proper transform of $S_2$ and $s_7$  defines  $\pi^{-1}(B)$;
\item for a very general $X$ as before we have an isomorphism 
\[
R(X)\to \cc[T_1,\dots, T_7]/I,\,\,s_i\mapsto T_i,
\]
where the degrees of the generators $T_i$ for $i=1, \dots, 7$ are given by the columns of the following matrix
%the Cox ring of a very general $X$ as before is isomorphic to $\cc[T_1,T_2,T_3,T_4,T_5,T_6,T_7,T_8,T_9]/I$, where the degrees of the generators are given by the columns of the following matrix
\[
\left(
\begin{array}{ccccccccc}
\,\,\,\,0& -1 &\,\,\,\,1 & -1 & \,\,\,\,0 & -2&-3  \\
\,\,\,\,0& \,\,\,\,0 & -1 &\,\,\,\, 0 &\,\,\,\,0 & -2 & -3 \\
\,\,\,\,0& \,\,\,\,1 & \,\,\,\,0 &\,\,\,\, 0 & -1 &\,\,\,\,1&\,\,\,\,1\\
-1&\,\,\,\,0&\,\,\,\, 0 & \,\,\,\,1 & \,\,\,\,0 &\,\,\,\, 1 & \,\,\,\,1  
\end{array}
\right)
\]
and the ideal $I$ is generated by the following polynomial
\begin{eqnarray*}
T_7^2-\tilde f(T_1,\dots,T_6),
\end{eqnarray*}
where $\tilde f(T_1,\dots,T_6)=f(T_1^2,T_2,\dots,T_6)$ and $f$ is the defining polynomial 
of $B$ in the Cox ring of $Y$.
% \item 
%the Cox ring of  $R(X)$ is isomorphic to $\cc[T_1,\dots,T_7]/(T_7^2-F(T_1,\dots,T_6))$,
%where $T_1,\dots,T_5$ correspond to the defining sections of the $(-2)$-curves, 
%$T_6$ is the defining section of the preimage of the proper transform of $S_2$
%and $T_7$ is the defining section of the component of positive genus of the ramification divisor.
\end{enumerate}
\end{proposition}
 
% In  \cite[Proposition 6.7, ii)]{A.H.L} $X$ is also described as the double cover 
%of a smooth quadric surface blown-up at two general points.
%\michela{agregar modelo en articulo de Rolleau?}

%\begin{proposition}\label{gen v4}
%Let $X$ be a K3 surface with $\Cl(X)\cong  V_4=U \oplus 2A_1$. Then
%\begin{enumerate}
%\item $X$ is a double cover  $\pi:X\to Y$, where $Y$ is a blow-up of $\ff_4$ at two general points $p,q$, 
%branched along $S_1$ and smooth curve of genus $8$;
%\item the surface has five $(-2)$-curves, which are the preimages by $\pi$ of: $S_1$,
%the proper transforms of  $F_1,F_2$ and the two exceptional divisors over $p$ and $q$;
% \item 
%the Cox ring of  $R(X)$ is isomorphic to $\cc[T_1,\dots,T_7]/(T_7^2-F(T_1,\dots,T_6))$,
%where $T_1,\dots,T_5$ correspond to the defining sections of the $(-2)$-curves, 
%$T_6$ is the defining section of the preimage of the proper transform of $S_2$
%and $T_7$ is the defining section of the component of positive genus of the ramification divisor.
%\end{enumerate}
%\end{proposition}
%
\begin{remark} \label{F2}
%A different model for $X$ can be obtained as follows.
The class $\BNef[3]$ has self-intersection $4$, is nef and has zero 
intersection with $f_2,f_3,f_4$. 
Moreover, it has intersection $2$ with the class $\BNef[5]$ of the 
unique elliptic fibration of $X$, thus it is hyperelliptic. 
It can be easily checked that $\BNef[3]$ is not equal to $3\BNef[5]+f_3$,
thus it is base point free by Theorem  \ref{bl}.
Moreover, $\BNef[3]=2\BNef[5]+2f_3+f_2+f_4$, thus
by \cite[Proposition 5.7, iii)]{SD} the image of the associated 
degree two morphism $\varphi: X\to \pp^3$ is a quadric cone. 
The morphism $\varphi$ factorizes through a double cover of the Hirzebruch surface 
$\eta:X\to\mathbb{F}_{2}$, and is branched along a divisor of the form 
$S+B$, where $S$ is the unique curve in $\mathbb F_2$
with $S^{2}=-2$ and $B\in|3S+8F|$ (here $F$ is a fiber). We denote
by $p,q$ the two intersection points of $B$ and $S$. 
The curves of classes $f_{1},f_{5}$ are mapped by $\eta$ to the two fibers 
of $\mathbb F_2$ through $p,q$, the curve of class $f_{3}$ has image $S$, and the curves
of classes  $f_{2}, f_{4}$ are contracted to the points $p,q$. 
Observe that this model is related to the one described in Proposition \ref{gen v4} as follows: 
the blow-up of $\mathbb F_2$ at $p,q$ is isomorphic 
to the blow-up of $\mathbb F_4$ at two general points, obtained 
contracting the proper transforms of the two fibers of $\mathbb F_2$ through $p,q$.
%The curves $A_{1},\dots,A_{4}$ generate the N\'eron-Severi group.
%The divisor
%\[
%D_{20}=(5,6,3,1)unique elliptic fibraion
%\]
% has square $20$, $D_{20}A_{1}=D_{20}A_{5}=2$ and $D_{20}A_{j}=1$
%for $j\in\{2,3,4\}$. The divisors
%\[
%A_{1}+A_{2},\,\,A_{4}+A_{5}
%\]
% are singular fibers of a fibration for which $A_{3}$ is a section.
%The divisor 
%\[
%D_{4}=2A_{1}+3A_{2}+2A_{3}+A_{4}
%\]
%has square $4$, is nef, base point free and $D_{4}A_{1}=D_{4}A_{5}=2$
%and $D_{4}A_{j}=0$ for $j\in\{2,3,4\}.$ Since there is an elliptic
%pencil $F=A_{1}+A_{2}$ such that $FD_{4}=2$, the linear system $|D_{4}|$
%is hyperelliptic. According to Theorem \cite[Proposition 5.7 and its proof]{SD}, case
%a) iii) v), the image of the double cover is a quadric cone in $\pp^{3}$.
%The double cover factorizes through the Hirzebruch surface $\eta:X\to\mathbf{F}_{2}$,
%and is branched over a curve $s+B$, where $s$ is the unique section
%with $s^{2}=-2$ and $B\in|3s+8f|$ (here $f$ is a fiber). We denote
%by $p,q$ the two intersection points of $B$ and $s$ (since $Bs=2$
%). The curves $A_{1},A_{5}$ are mapped by $\eta$ to the two fibers
%through $p,q$, the curve $A_{3}$ has image $s$, and the curves
%$A_{2},A_{4}$ are contracted to the points $p,q$. 
\end{remark}

\begin{corollary} \label{uni4}
The moduli space of K3 surfaces $X$ with $\NS(X)\simeq U\oplus 2{A}_{1}$
is unirational.
\end{corollary}

\begin{proof}
By Remark \ref{F2} we know that $X$ is the double cover of $\mathbb F_{2}$ branched 
along a  curve of type $B+S$, where $S^2=-2$ and $B\in |3S+8F|$ is smooth.
By \cite[Chapter V, Corollary 2.18]{RH}, the divisors $B$
and $B-K_{\ff_2}=5S+12F$ are ample. Thus $h^{2}(\ff_2,B)=0$ 
and by Serre duality and Kodaira vanishing Theorem $h^{1}(\ff_2,B)=h^{1}(\ff_2,K_{\ff_2}-B)=0$.
%\]
It follows by the Riemann-Roch Theorem
%The canonical divisor of $\mathbb F_{2}$ is 
%\[
%K_{\ff_{2}}=-(2S+4F),
%\]
%thus by Riemann-Roch Theorem on $\ff_{3}$:
%\[
%\chi(B)=\tfrac{1}{2}B(B-K)+1=24.
%\]
%By \cite[Chapter V, Corollary 2.18]{Hartshorne}, the divisors $B$
%and $B-K=5s+12f$ are ample. Thus one has $h^{2}(X,\Osh(B))=0$ and
%by Serre duality and Kodaira vanishing Theorem:
%\[
%h^{1}(X,\Osh(B))=h^{1}(X,\Osh(K-B))=0.
%\]
that $h^{0}(\ff_2,B)=24.$ The dimension of the
automorphism group of $\ff_{2}$ is $7$, thus the quotient of $|B|\simeq\pp^{23}$
by that automorphism group is a $16$ dimensional unirational variety
which is birational to the moduli space of K3 surfaces $X$ with $\NS(X)\simeq U\oplus 2A_{1}$.
\end{proof}

 \subsection*{The family $\mathcal F_5$}

Let $X$ be a K3 surface with $\NS(X)\cong V_5=U(2) \oplus 2A_1$.
By Theorem \ref{main-eff} $X$ contains six classes of $(-2)$-curves 
whose intersection matrix is given in Table \ref{Intmat} and described in Figure \ref{fig3}.
The Hilbert basis of the nef cone of $X$ contains five classes.
The classes $\BNef[3]$, $\BNef[4]$ and $\BNef[5]$ are classes of elliptic fibrations without sections and 
each having two reducible fibers of type $\tilde A_1$. 

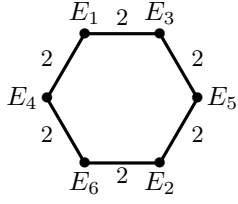
\begin{figure}[H]
\begin{center}
\begin{tikzpicture}[scale=1]

\draw [very thick] (1,0) -- (0.5,0.86);
\draw [very thick] (0.5,0.86) -- (-0.5,0.86);
\draw [very thick] (-1,0) -- (-0.5,0.86);
\draw [very thick] (-1,0) -- (-0.5,-0.86);
\draw [very thick] (-0.5,-0.86) -- (0.5,-0.86);
\draw [very thick] (1,0) -- (0.5,-0.86);

\draw (1,0) node {$\bullet$};
\draw (0.5,0.86) node {$\bullet$};
\draw (-1,0) node {$\bullet$};
\draw (-0.5,0.86) node {$\bullet$};
\draw (-0.5,-0.86) node {$\bullet$};
\draw (0.5,-0.86)  node {$\bullet$};

\draw (1,0) node  [right]{$E_{5}$};
\draw (0.5,0.86) node [above]{$E_{3}$};
\draw (-1,0) node  [left]{$E_{4}$};
\draw (-0.5,0.86) node  [above]{$E_{1}$};
\draw (-0.5,-0.86) node  [below]{$E_{6}$};
\draw (0.5,-0.86) node   [below]{$E_{2}$};

\draw (1,0.3) node [above]{\small $2$};
\draw (1,-0.7) node [above]{\small $2$};
\draw (0,-0.8) node [below]{\small $2$};
\draw (-1,-0.7) node [above]{\small $2$};
\draw (-1,0.3) node [above]{\small $2$};
\draw (0,1.3) node [below]{\small $2$};

\end{tikzpicture}
\end{center}
\caption{Intersection graph of $(-2)$-curves of $\mathcal F_5$} \label{fig3}
\end{figure}

\begin{proposition}\label{gen v5}
Let $X$ be a K3 surface with $\NS(X)\cong  V_5=U(2) \oplus 2A_1$. Then
\begin{enumerate}
\item there is a minimal resolution $\varphi:X\to \overline X$ of a double cover  
$\pi:\overline X\to \pp^2$ branched along a plane sextic $B$ with three nodes $p_1,p_2,p_3$;
\item the surface $X$ has six $(-2)$-curves: the exceptional divisors $E_1,E_2,E_3$ over the three nodes $p_1,p_2,p_3$ 
and the curves $R_1,R_2,R_3$ such that $\pi\varphi(R_i)$ is the line through $p_j,p_k$ with $j,k\not=i$;
\item a minimal generating set of the Cox ring of  $X$ is $s_1,\dots,s_7$, where  $s_1, \dots, s_{6}$ define the $(-2)$-curves and $s_7$ defines the ramification curve of $\pi\varphi$;
\item for a very general $X$ as before we have an isomorphism 
\[
 \cc[T_1,\dots, T_7]/I\to R(X),\quad T_i\mapsto s_i,
\]
where the degrees of the generators $T_i$ for $i=1, \dots, 7$ are given by the columns of the following matrix
%the Cox ring of a very general $X$ as before is isomorphic to $\cc[T_1,T_2,T_3,T_4,T_5,T_6,T_7,T_8,T_9]/I$, where the degrees of the generators are given by the columns of the following matrix
\[
\left(
\begin{array}{ccccccccc}
\,\,\,\,0& \,\,\,\,0 & \,\,\,\,0 & -1 & -1 &\,\,\,\, 0& -2 \\
-1& -1 &\,\,\,\, 0 &\,\,\,\, 0 &\,\,\,\, 0 &\,\,\,\, 0 & -2\\
\,\,\,\,0&\,\,\,\, 1 &\,\,\,\, 0 & \,\,\,\,1& \,\,\,\,0& -1& \,\,\,\,1\\
\,\,\,\,1&\,\,\,\, 0& -1& \,\,\,\,0& \,\,\,\,1& \,\,\,\,0&\,\,\,\,1 
\end{array}
\right)
\]
and the ideal $I$ is generated by the following polynomial:
\begin{eqnarray*}
T_7^2-f(T_1T_3T_5, T_2T_5T_6,T_1T_4T_6),
\end{eqnarray*}
where $f(x_0,x_1,x_2)\in \cc[x_0,x_1,x_2]$ is a defining polynomial of the plane sextic $B$.
% \item 
%the Cox ring of  $R(X)$ is isomorphic to $\cc[T_1,\dots,T_7]/(T_7^2-F(T_1,\dots,T_6))$,
%where $T_1,\dots,T_5$ correspond to the defining sections of the $(-2)$-curves, 
%$T_6$ is the defining section of the preimage of the proper transform of $S_2$
%and $T_7$ is the defining section of the component of positive genus of the ramification divisor.
\end{enumerate}
% \item 
%the Cox ring of  $X$ is isomorphic to $\cc[T_1,\dots,T_7]/(T_7^2-F(T_1,\dots,T_6))$,
%where $T_1,T_2,T_3$ define $E_1,E_2,E_3$, $T_4,T_5,T_6$ define $R_1,R_2, R_3$,
%%$T_1,\dots,T_6$ correspond to the defining sections of the $(-2)$-curves 
%$T_7$ to the defining section of the ramification curve, and define $F(T_1,\dots,T_6)=f (T_1T_2T_6, T_1T_3T_5,T_2T_3T_4)$, where $f$ is the define polygonal of the plane sextic $B$.
%%where $s_i$ be a generator of $H^0(f_i)$ for all i, and a basis $x_0, x_1, x_2\in H^0 (h)$. 
%\end{enumerate}
\end{proposition}

\begin{proof}
Let $f_1,\dots,f_5$ be the classes of the $(-2)$-curves (see Table \ref{Table1})  and let  $h=\BNef[1]$. 
%$h_2=\BNef[2]$  
Then  $h^2=2$, $h\cdot f_i=0$ for $i=1,5,6$ and $h\cdot f_i=2$ for $i=2, 3,4$. 
Thus $h$ is nef. By Corollary \ref{section} the linear system associated to $h$ is base point free and 
thus defines a degree two morphism to $\pp^2$ branched along a plane sextic $B$ which contracts 
the $(-2)$-curves of classes $f_1, f_5, f_6$. Since such classes have zero intersection, then 
the three $(-2)$-curves are disjoint, thus are mapped to three distinct points $p_1,p_2,p_3\in \pp^2$. 
The plane sextic $B$ has three nodes at $p_1,p_2,p_3$.
Moreover, since $h=f_1+f_3+f_5=f_2+f_5+f_6=f_1+f_4+f_6$, the $(-2)$-curves 
with classes $f_2,f_3,f_4$ are mapped to the three lines passing through $p_2,p_3$,
$p_1,p_2$ and $p_1,p_3$ respectively.

%\begin{itemize}
%\item $h_1\cdot f_i=0$ for $i=1,\ 5,\ 6$ and $h_1\cdot f_i=2$ for $i=2,\ 3,\ 4$
%\item $h_2\cdot f_i=0$ for $i=2,\ 3,\ 4$ and $h_2\cdot f_i=2$ for $i=1,\ 5,\ 6$ 
%\item $h^0(h_i)=3$ and $h^1(h_i)=0$ for $i=1,\ 2$
%\item $h_1=f_1+f_3+f_5=f_2+f_5+f_6=f_1+f_4+f_6$ 
%\item $h_2=f_1+f_3+f_4=f_2+f_4+f_6=f_2+f_3+f_5$.
%\end{itemize}

%The general element of $\mathcal F_5$ is the minimal resolution 
% $\varphi:X\to Y$ of a double cover  $\pi:Y\to \pp^2$
% cover of $\pp^2$ branched along a plane sextic with three nodes $p_1,\ p_2,\ p_3$. 
% The surface has six $(-2)$-curves: the exceptional divisors $E_1,\ E_2,\ E_3$ 
%over the singular points $\pi^{-1}(p_i)$ and three curves $R_1,\ R_2,\ R_3$ such that $\pi\varphi(R_i)$ is the line through $p_j,\ p_k$ with $j,k\not=i$.

The Cox ring of $X$ has been computed in \cite[Proposition 6.7, i)]{A.H.L},
where $X$ is also described as the double cover of 
 a smooth quadric surface $\ff_0$ blown-up at two general points.
 \end{proof}

 \begin{corollary}\label{uni5}
The moduli space of K3 surfaces with $\NS(X)\simeq U(2)\oplus 2{A}_{1}$
is unirational.
\end{corollary}

\begin{proof}
By Proposition \ref{gen v5} $X$ is the minimal resolution of a 
double cover of $\pp^2$ branched along a plane sextic $B$ with three nodes,
thus the moduli space of such K3 surfaces is birational to the moduli space
of such plane sextics. 
Up to a projectivity we can assume that the three nodes of $B$ are 
the fundamental points of $\pp^2$. The linear system of plane sextics 
with those nodes has dimension $18$. The quotient of such space by the action of the projectivities 
leaving invariant the set of fundamental points 
gives a $16$-dimensional unirational moduli space.
%The $16$-dimensional moduli space of K3 surfaces with $\NS X)\simeq U(2)\oplus\mathbf{A}_{1}^{\oplus2}$
%is therefore unirational.
\end{proof}

\subsection*{The family $\mathcal F_6$}

Let $X$ be a K3 surface with $\NS(X)\cong V_6=U(3)\oplus 2A_1$.
By Theorem \ref{main-eff}, $X$ contains eight $(-2)$-curves 
whose intersection matrix is given in Table \ref{Intmat}.
The Hilbert basis of the nef cone of $X$ contains $19$ classes, 
four of them are classes of elliptic fibrations: $\BNef[5]$, $\BNef[8]$, $\BNef[18]$ and $\BNef[19]$. 
Each elliptic fibration has two fibers of type $\tilde A_1$ and has no sections. 

\begin{proposition}\label{gen v6}
Let $X$ be a K3 surface with $\NS(X)\cong V_6=U(3)\oplus 2A_1$. Then
\begin{enumerate}
\item  there is a double cover $\pi:X\to \pp^2$ branched along 
 a smooth plane sextic with two $3$-tangent lines  $L_1, L_2$  and two $6$-tangent conics $C_1, C_2$;
 \item  $X$ can be defined by an equation of the following form in $\pp(1,1,1,3)$:
 \[
x_3^2= F_1(x_0,x_1,x_2)F_2(x_0,x_1,x_2)G_1(x_0,x_1,x_2)G_2(x_0,x_1,x_2)+F(x_0,x_1,x_2)^2,
 \]
 where  $F\in\cc[x_0,x_1, x_2]$ is homogeneous of degree three, $G_1,G_2\in\cc[x_0,x_1, x_2]$ are homogeneous of degree two 
 and $F_1, F_2\in\cc[x_0,x_1, x_2]$ are homogeneous of degree one;
 \item the surface has eight $(-2)$-curves:  the four curves $R_{ij}$, $i,j=1,2$ such that 
$\pi(R_{i1})=\pi(R_{i2})=L_i$ and the four curves $S_{i1}, S_{i2}$ such that 
$\pi(S_{i1})=\pi(S_{i2})=C_i$ for $i=1,2$;
 \item a minimal generating set of the Cox ring of  $X$ is $s_1,\dots,s_9$, where 
  $s_1, \dots, s_{8}$ define the $(-2)$-curves and  $s_9\in H^0(\pi^*\oo_{\pp^2}(1))$;
%where $s_i$ be a generator of $H^0(f_i)$ for all i, and a basis $x_0, x_1, x_2\in H^0 (h)$. 
%where the degrees of the generators are given by the columns of the following matrix
%\[
%\left(
%\begin{array}{ccccccccc}
%0& -2 & 0 & 0 & -1 & -2 & 0 & -1 & -1\\
%0& -2 & -1 & 0 & 0 & -2 & -1 & 0 & -1\\
%0& -3 & 0 & 1 & 0 & -2 & -1 & -1 & -1\\
%1& -2 & -1 & 0 & -1 & -3 & 0 & 0 & -1
%\end{array}
%\right).
%\]
\item 
for a very general $X$ as before we have an isomorphism 
\[
 \cc[T_1,\dots, T_9]/I\to R(X),\quad T_i\mapsto s_i,
\]
where the degrees of the generators $T_i$ for $i=1, \dots, 9$ are given by the columns of the following matrix
%the Cox ring of a very general $X$ as before is isomorphic to $\cc[T_1,T_2,T_3,T_4,T_5,T_6,T_7,T_8,T_9]/I$, where the degrees of the generators are given by the columns of the following matrix
\[
\left(
\begin{array}{ccccccccc}
\,\,\,\,0& -2 &\,\,\,\, 0 & \,\,\,\,0 & -1 & -2 &\,\,\,\, 0 & -1 & -1\\
\,\,\,\,0& -2 & -1 & \,\,\,\,0 &\,\,\,\, 0 & -2 & -1 &\,\,\,\, 0 & -1\\
\,\,\,\,0& -3 & \,\,\,\,0 &\,\,\,\, 1 & \,\,\,\,0 & -2 & -1 & -1 & -1\\
\,\,\,\,1& -2 & -1 &\,\,\,\, 0 & -1 & -3 &\,\,\,\, 0 &\,\,\,\, 0 & -1
\end{array}
\right)
\]
and the ideal $I$ is generated by the following polynomials:
\begin{eqnarray*}
&T_1T_6-\tilde G_1(T_3T_8,T_5T_7,T_9),\\
&T_2T_4-\tilde G_2(T_3T_8,T_5T_7,T_9),\\
&T_1T_2T_3T_5+T_4T_6T_7T_8-\tilde F(T_3T_8,T_5T_7,T_9),
\end{eqnarray*}
%where $\tilde G_1, \tilde G_2$ are homogeneous of degree two 
%and $\tilde F$ of degree three.
where $\tilde G_i=G\circ \varphi^{-1}$, $i=1,2$, and $\tilde F=F\circ \varphi^{-1}$ 
are obtained from $G_1,G_2$ and $F$ in item ii) respectively composing with 
the coordinate change in $\pp^2$ given by 
\[
\varphi(x_0,x_1,x_2)=(F_1, F_2, s_9).
\]
\end{enumerate}
\end{proposition}

\begin{proof}
Let $f_1,\dots, f_8$ be the classes of the $(-2)$-curves (see Table \ref{Table1}) 
and $h=\BNef[15]$. Then  $h^2=2$, 
 $h\cdot f_i=2$ for $i=1,2,4,6$ and $h\cdot f_i=1$ for $i=3,5,7,8$.
 Thus $h$ is ample. By Corollary \ref{section} the associated linear system is base point free 
 and thus defines a double cover $\pi:X\to \pp^2$ branched along a smooth plane sextic $B$.
Since $h=f_3+f_{8}=f_5+f_{7}$ and $2h=f_1+f_6=f_2+f_4$ 
the image by $\pi$ of the eight $(-2)$-curves of $X$ are two $3$-tangent lines  $L_1,L_2$  
and two $6$-tangent conics $C_1,C_2\subseteq \pp^2$. This proves items i) and iii).

To prove ii), observe that, looking at the intersection matrix of the $(-2)$-curves of $X$ one can see that 
the restriction of the double cover $\pi$ over $D=L_1\cup L_2\cup C_1\cup C_2$ is trivial.
This implies that there exists a plane cubic $C$ such that $B\cdot D=2C\cdot D$ (see for example \cite[Proposition 1.7, Ch.3]{Vermeulen}).
Let $F_1, F_2, G_1, G_2, F$ be defining polynomials for $L_1, L_2, G_1, G_2$ and $F$  
respectively. Consider the pencil 
of plane sextics generated by $F_1F_2G_1G_2$ and $F^2$, which contains $B\cap D$ in its base locus. 
Let $V$ be an element of the pencil intersecting $B$ in one more point. By Bezout's Theorem, since $V$ intersects $B$ in at least $37$ points counting multiplicities and $B$ is irreducible, we have that $V=B$.

By Theorem \ref{main-cox} the Cox ring $R(X)$ is generated in the following degrees: 
\[
f_1,\dots, f_{8}, h. 
\]
Clearly any minimal generating set of $R(X)$ must contain the sections $s_1,\dots, s_{8}$ defining 
the $(-2)$-curves of $X$ and a section $s_9$ such that $s_3s_8, s_5s_7, s_9$ is a basis of $H^0(h)$.
This proves item iv).

The first two relations in item v) are obvious, due to the fact that $s_1s_6$ and $s_2s_4$ define the preimages of the conics $C_1=\{G_1=0\}$ and $C_2=\{G_2=0\}$. 
%Observe that $h^0(3h)=11$, $\dim({\rm Sym}^3H^0(h))=10$ and $s_1s_2s_3s_5, s_4s_6s_7s_8\in H^0(3h)$.
%Moreover, ${\rm Sym}^3H^0(h)$ is the invariant subspace of $H^0(3h)$ for the natural action of the covering involution $i$ of $\pi$.
%Since $s_1s_2s_3s_5+ s_4s_6s_7s_8$ is invariant, then it belongs to ${\rm Sym}^3H^0(h)$. 
The last relation, follows from the fact that 
\[
(x_3+F )(x_3-F)=F_1F_2G_1G_2,
\] 
thus up to renumbering we can assume $x_3+F=s_1s_2s_3s_5$, $x_3-F=s_4s_6s_7s_8$, so that $2F=s_1s_2s_3s_5 - s_4s_6s_7s_8$. Up to rescaling the generators $s_i$ we obtain the last relation.
It can be proved with the same type of argument used in the proof of \cite[Theorem 3.5]{A.C.L} that the ideal $I$ is prime for general $F_1, F_2, G_1, G_2, F$. 
Since the ring $\cc[T_1,\dots, T_9]/I$ is an integral domain, it has Krull dimension $6$ and it surjects onto $R(X)$, 
then $\cc[T_1,\dots,T_9]/I\cong R(X)$.
\end{proof}

 \begin{corollary}\label{uni6}
The moduli space of K3 surfaces with $\NS(X)\simeq U(3)\oplus 2{A}_{1}$
is unirational.
\end{corollary}

\begin{proof}
It follows from item ii) in Proposition \ref{gen v6} that  $X$ is a double cover of $\pp^2$ branched 
along a plane sextic which can be defined 
by an equation of the form $F_1F_2G_1G_2+F^2=0$,
 where $F_1,F_2\in\cc[x_0,\dots, x_2]$ are homogeneous polynomials of degree one, 
$G_1,G_2\in\cc[x_0,\dots, x_2]$ of degree two and $F\in \cc[x_0,\dots,x_2]$ of degree three.
 
 Conversely, a double cover $Y$ of $\pp^2$ branched along a 
 smooth plane sextic $B$ with an equation of that form 
 has the property that the curves defined by $F_i=0$ for $i=1,2$ or $G_j=0$ 
 for $j=1,2$ are  everywhere tangent to $B$. Moreover, the double cover 
 is trivial over their union. This implies that the pull-back of each such curve 
 is the union of two smooth rational curves and it can be easily proved that 
 the classes of such curves generate a sublattice of $\NS(Y)$ isometric to $V_6$.
 % for any $i=1,2,3,4$ its intersection with the plane
% $F_i=0$ is the union of two conics $C_{i1}, C_{i2}$ defined by $G_1=0$ and $G_2=0$.
%% Moreover, intersecting $Y$ with the quadric $q_1=0$ one obtains 
%% four conics $C_1,\dots,R_4$ lying on the hyperplanes $\ell_1=0,\dots, \ell_4=0$ 
%% respectively.
%In particular, for a general choice of the polynomials $G_1,G_2,F_1,\dots,F_4$ 
%we have $C_{i1}\cdots C_{i2}=4$, $C_{ij}\cdot C_{i'j}=2$ for $i\not=i'$ and $j=1,2$ 
%and $C_{ij}\cdot C_{i'j'}=0$ for $i\not=i'$ and $j\not=j'$.
%An easy computation shows that the classes of $C_{11}+C_{21}, C_{12}+C_{32}, C_{21}, C_{32}$ generate 
%a sublattice of the Picard lattice of $Y$ isometric to $V_7$. 
Thus the moduli space of such K3 surfaces $Y$ has dimension at most $20-4=16$.
%Thus a general member of $\mathcal F$belongs to the family $\mathcal F_7$.

This shows that an open subset of the moduli space of K3 surfaces with $\NS(X)\simeq U(3)\oplus 2A_{1}$
is isomorphic to the quotient by ${\rm PGL}(3,\cc)$ of an open subset of the image of the map:
\[
\Phi:H^{0}(\pp^{2},\Osh_{\pp^2}(1))^{\oplus2}\oplus H^{0}(\pp^{2},\Osh_{\pp^2}(2))^{\oplus2}\oplus  H^{0}(\pp^{2},\Osh_{\pp^2}(3))
\to H^{0}(\pp^{2},\Osh_{\pp^2}(6))
\]
defined by 
\[
(F_1,F_2, G_1,G_2,F)\to  F_1F_2G_1G_2+F^2.
\]
In particular, it is unirational.
\end{proof}

\begin{example}
Let us take 
\[
\begin{array}{c}
F_{1}=x_1+x_2,\ F_{2}=2(x_0+x_1+x_2)\\[1.5pt]
G_{1}=x_0^{2}+x_0x_2+x_1x_2+x_2^{2},\ G_{2}=x_0^{2}+x_0x_1+2x_1^{2}+x_2^{2}\\[1.5pt]
F=2x_0^{2}x_1+2x_0x_1^{2}+2x_1^{3}+2x_0x_1x_2+2x_1^{2}x_2+x_2^{3}.
\end{array}
\]
One can check that the plane sextic defined by $F_{1}F_{2}G_{1}G_{2}-F^{2}=0$
is smooth. Let $Y$ be the double cover of $\pp^2$ branched along it. The 
reduction $Y_{3}$ of $Y$ modulo $3$ is smooth. One computes
that the Weil polynomial of the Frobenius action on the second \'etale
cohomology group of $Y_{3}$ has the form 
\[
P_{2}=(T-3)(T+3)^{3}Q_{2},
\]
where $Q_{2}$ has no roots $\alpha$ such that $\frac{\alpha}{3}$ is a root
of unity. Therefore the geometric Picard number of $Y$ is $4$ and
$Y_{/\cc}$ is a K3 surface with $\NS(Y_{/\cc})\simeq U(3)\oplus 2A_{1}$. 
\end{example}

\subsection*{The family $\mathcal F_7$}

Let $X$ be a K3 surface with $\NS(X)\cong V_7=U(4)\oplus 2A_1$.
By Theorem \ref{main-eff} $X$ contains eight $(-2)$-curves 
and  the Hilbert basis of the nef cone of $X$ contains $15$ classes, 
with six classes of elliptic fibrations $\BNef[4]$, $\BNef[7]$, $\BNef[9]$, $\BNef[13]$, $\BNef[14]$ and $\BNef[15]$. 
Each elliptic fibration is without sections and has two fibers of type $\tilde A_1$.

\begin{proposition}\label{gen v7}
Let $X$ be a K3 surface with $\NS(X)\cong V_7=U(4)\oplus 2A_1$. Then
\begin{enumerate}
\item $X$ is isomorphic to a smooth quartic 
 surface in $\pp^3$ having four hyperplane sections which are the union of two conics;
 \item  $X$ can be defined by an equation of the form
 \[
 G_1(x_0,\dots,x_3)G_2(x_0, \dots,x_3)+F_1(x_0, \dots,x_3)F_2(x_0, \dots,x_3)F_3(x_0, \dots,x_3)F_4(x_0, \dots,x_3)=0,
 \]
 where $G_1,G_2\in\cc[x_0,\dots, x_3]$ are homogeneous of degree two 
 and $F_i\in\cc[x_0,\dots, x_3]$ are homogeneous of degree one 
 for $i=1,2,3,4$;
\item the surface has eight $(-2)$-curves: the eight conics;
 \item a minimal generating set of the Cox ring of  $X$ is $s_1,\dots, s_8$, where  $s_1,\dots, s_{8}$ define 
 the $(-2)$-curves (we assume that $s_6s_7$, $s_3s_5$, $s_1s_8$ and $s_2s_4$ define the four reducible hyperplane sections);
%where $s_i$ be a generator of $H^0(f_i)$ for all i, and a basis $x_0, x_1, x_2\in H^0 (h)$. 
%where the degrees of the generators are given by the columns of the following matrix
%\[
%\left(
%\begin{array}{cccccccc}
%0& 0& -1& -1 & 0& -1 & 0 & -1 \\
%-1& 0 & -1 & -1 & 0 & 0 & -1 & 0 \\
%0& 0 & -2 & -1 & 1 & 0 & -1 & -1\\
%1& -1 & 1 & 2& 0& 1& 0 & 0 
%\end{array}
%\right).
%\]
\item for a very general $X$ as before we have an isomorphism 
\[
 \cc[T_1,\dots, T_8]/I\to R(X),\quad T_i\mapsto s_i,
\]
where the degrees of the generators $T_i$ for $i=1, \dots, 8$ are given by the columns of the following matrix
%the Cox ring of a very general $X$ as before is isomorphic to $\cc[T_1,T_2,T_3,T_4,T_5,T_6,T_7,T_8]/I$, where the degrees of the generators are given by the columns of the following matrix
\[
\left(
\begin{array}{ccccccccc}
\,\,\,\,0& \,\,\,\,0 & -1 & -1& \,\,\,\,0 & -1 &  \,\,\,\,0& -1 \\
-1& \,\,\,\,0& -1 & -1&\,\,\,\, 0 & \,\,\,\,0 & -1 & \,\,\,\,0 \\
\,\,\,\,0& \,\,\,\,0 & -2 & -1&\,\,\,\,1 &\,\,\,\, 0 & -1 & -1 \\
\,\,\,\,1& -1 & \,\,\,\,1 &\,\,\,\, 2&\,\,\,\, 0& \,\,\,\,1& \,\,\,\,0 &\,\,\,\, 0 
\end{array}
\right)
\]
and the ideal $I$ is generated by the following polynomials:
\begin{eqnarray*}
&T_1T_2T_3T_6-\tilde G_1(T_6T_7,T_3T_5,T_1T_8,T_2T_4),\\
&T_4T_5T_7T_8-\tilde G_2(T_6T_7,T_3T_5,T_1T_8,T_2T_4),
\end{eqnarray*}
where $\tilde G_i=G\circ \varphi^{-1}$, $i=1,2$  
are obtained from $G_1,G_2$ in item ii) respectively composing with 
the coordinate change in $\pp^3$ given by 
\[
\varphi(x_0,x_1,x_2,x_3)=(F_1, F_2, F_3, F_4).
\]
\end{enumerate}
\end{proposition}

\begin{proof}
Let $f_1,\dots,f_8$ be the classes of the $(-2)$-curves  (see Table \ref{Table1}) and let $h=\BNef[11]$.
Then $h^2=4$ and $h\cdot f_i=2$ for all $i$. Thus $h$ is ample. 
Moreover, $h$ is not hyperelliptic by Proposition \ref{hyp}.
By Corollary \ref{section} the associated linear system  is base point free,
thus it defines an embedding of $X$ in $\pp^3$ as a smooth quartic surface.
Since $h=f_1+f_{8}=f_2+f_{4}=f_3+f_5=f_6+f_7$ and $h\cdot f_i=2$ for all $i$, then 
$X$ has four hyperplane sections which decompose into the union of two conics.
We will denote by $C_i$ the conic whose class is $f_i$.
Observe that $f_1+f_2+f_3+f_6=f_4+f_5+f_7+f_8=2h$. This means 
that the four conics $C_1, C_2, C_3, C_6$ are contained in a quadric $Q_1$. 
For the same reason the conics $C_4, C_5, C_7, C_8$ are contained in a quadric $Q_2$. 

Let $F_1,\dots, F_4$ be defining polynomials for the four planes containing the 
conics $C_1,\dots, C_8$, and $G_1, G_2$ be defining polynomials for the 
quadrics $Q_1,Q_2$ respectively. Consider the pencil of quartic surfaces 
in $\pp^3$ generated by $G_1G_2$ and $F_1F_2F_3F_4$. 
Any member of the pencil intersects $X$ along the degree $16$ curve 
$C_1\cup\cdots\cup C_8$. 
Since $X$ is irreducible, it follows from Bezout's theorem that $X$ belongs to the pencil.

 By Theorem \ref{main-cox} the Cox ring $R(X)$ is generated in the following degrees: 
\[
f_1,\dots, f_{8},\ h. 
\]
Clearly any minimal generating set of $R(X)$ must contain the sections $s_1,\dots, s_{8}$ defining 
the $(-2)$-curves of $X$.  
Since generically the four polynomials $F_1,\dots,F_4$ 
can be taken to be independent,  then they generate $H^0(\oo_{\pp^3}(1))$.
This implies that  $s_1s_8, s_2s_4, s_3s_5, s_6s_7$ generate $H^0(h)$,
so that a generator in degree $H^0(h)$ is not necessary. This proves item iv).
The two relations in item v) are due to the fact that $s_1s_2s_3s_6$ and $s_4s_5s_7s_8$ define the intersections 
$X\cap Q_i$ for $i=1,2$ respectively.

It can be proved with the same type of argument used in the proof of 
\cite[Theorem 3.5]{A.C.L} that the ideal $I$ is prime for general $G_1, G_2$. 
Since $\cc[T_1,\dots,T_8]/I$ is an integral domain of dimension $\dim R(X)= \dim (X)+ \rk \Cl(X)=6$ 
and surjects onto $R(X)$, then it is isomorphic to it.
\end{proof}

\begin{corollary}\label{uni7}
The moduli space of K3 surfaces $X$ with $\NS(X)\simeq U(4)\oplus 2A_{1}$
is unirational. 
\end{corollary}

\begin{proof}
It follows from item ii) in Proposition \ref{gen v7} that  $X$ can be defined 
by an equation of the form $G_1G_2+F_1F_2F_3F_4=0$,
 where $G_1,G_2\in\cc[x_0,\dots, x_3]$ are homogeneous polynomials of degree two 
 and $F_i\in\cc[x_0,\dots, x_3]$ are homogeneous polynomials of degree one 
 for $i=1,2,3,4$. 
 
 Conversely, a smooth quartic surface $Y$ with an equation of that form 
 has the property that for any $i=1,2,3,4$ its intersection with the plane
 $F_i=0$ is the union of two conics $C_{i1}, C_{i2}$ defined by $G_1=0$ and $G_2=0$.
% Moreover, intersecting $Y$ with the quadric $q_1=0$ one obtains 
% four conics $C_1,\dots,R_4$ lying on the hyperplanes $\ell_1=0,\dots, \ell_4=0$ 
% respectively.
In particular, for a general choice of the polynomials $G_1,G_2,F_1,\dots,F_4$ 
we have $C_{i1}\cdots C_{i2}=4$, $C_{ij}\cdot C_{i'j}=2$ for $i\not=i'$ and $j=1,2$ 
and $C_{ij}\cdot C_{i'j'}=0$ for $i\not=i'$ and $j\not=j'$.
An easy computation shows that the classes of $C_{11}+C_{21}, C_{12}+C_{32}, C_{21}, C_{32}$ generate 
a sublattice of the Picard lattice of $Y$ isometric to $V_7$. 
This implies that the moduli space of such quartic surfaces has dimension at most $20-4=16$.
%Thus a general member of $\mathcal F$belongs to the family $\mathcal F_7$.

This shows that an open subset of the moduli space of K3 surfaces with $\NS(X)\simeq U(4)\oplus 2A_{1}$
can be obtained taking the quotient by ${\rm PGL}(4,\cc)$ of an open subset of the image of the map:
\[
\Phi:H^{0}(\pp^{3},\Osh_{\pp^3}(1))^{\oplus4}\oplus H^{0}(\pp^{3},\Osh_{\pp^3}(2))^{\oplus2}\to H^{0}(\pp^{3},\Osh_{\pp^3}(4))
\]
defined by 
\[
(F_1,F_2,F_3, F_4, G_1,G_2)\to  F_1F_2F_3F_4+G_1G_2.
\]
In particular, it is unirational.
\end{proof}

\subsection*{The family $\mathcal F_8$}

Let $X$ be a K3 surface with $\NS(X)\cong V_8=U\oplus A_2$.
By Theorem \ref{main-eff} $X$ contains four $(-2)$-curves 
whose intersection matrix is given in Table \ref{Intmat} and whose 
intersection graph is described in Figure \ref{fig4}.
The Hilbert basis of the nef cone of $X$ contains five classes.
The class $\BNef[5]$ defines an elliptic fibration  having a section 
and one fiber of type $\tilde A_2$. 

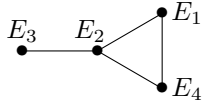
\begin{figure}[H]
\begin{center}
\begin{tikzpicture}[scale=1]

\draw (2,0) -- (3,0);
\draw (3,0) -- (3.86,0.5);
\draw (3,0) -- (3.86,-0.5);
\draw (3.86,0.5) -- (3.86,-0.5);

\draw (2,0) node {$\bullet$};
\draw (3,0) node {$\bullet$};
\draw (3.86,0.5) node {$\bullet$}; 
\draw (3.86,-0.5) node {$\bullet$}; 

\draw (2,0) node [above]{$E_{3}$};
\draw (2.9,0) node [above]{$E_{2}$};
\draw (3.86,0.5) node [right]{$E_{1}$};
\draw (3.86,-0.5) node [right]{$E_{4}$};

\end{tikzpicture}
\end{center} 
\caption{Intersection graph of $(-2)$-curves for $\mathcal F_8$}\label{fig4}
\end{figure}

\begin{proposition}\label{gen v8}
Let $X$ be a K3 surface with $\NS(X)\cong V_8=U\oplus A_2$. Then
\begin{enumerate}
\item there is a degree two morphism $\varphi:X\to \pp^4$, 
whose image is a cone over a rational normal cubic in $\pp^3$,
which factors through a degree two morphism 
$\mu:X\to \ff_3$ branched along 
the union of the smooth rational curve $E$  with $E^2=-3$ and   
a reduced curve $B$ intersecting $E$ at one point $p$ and 
the fiber of $\ff_3\to \pp^1$ through $p$ at one point with multiplicity two;
  
\item the surface has four $(-2)$-curves:  $\mu^{-1}(E)$, $\mu^{-1}(p)$ 
and two smooth rational curves mapping to the fiber of $\ff_3$ through $p$;

\item the Cox ring $R(X)$ is generated in the degrees given  in Table \ref{TableGen rank4}, in particular it 
has at least $8$ generators whose degrees are either classes of $(-2)$-curves 
or the classes $\BNef[i]$, for $i=1,2,3,5$, in the Hilbert basis of the nef cone.
\end{enumerate}
\end{proposition}

\begin{proof}
Let $f_1,\dots,f_4$ be the classes of the $(-2)$-curves $E_1,\dots,E_4$,
where $f_3$ is the class of the section of the elliptic fibration 
defined by $\BNef[5]$,
%and $h_1=\BNef[4]$. Then$h_1^2=2$, 
% $h_1\cdot f_2=1$, and $h_1\cdot f_i=0$ for $i\neq 2$. 
 %$h_1^0(h)=3$ and $h_1^1(h)=0$. 
and $h=\BNef[3]$. Then  $h^2=6$, 
 $h\cdot f_i=1$ for $i=1,4$ and $h\cdot f_i=0$ for $i=2,3$. 
 Since $h\cdot \BNef[5]=2$ then the associated linear system is hyperelliptic by Proposition \ref{hyp}.
 Moreover, it is base point free by Theorem \ref{bl}, since $h\not=4\BNef[5]+f_3$.
 Let $\varphi$ be the associated morphism.
% Since $h=3\BNef[5]+2f_3+f_2$ where $h\cdot f_1=1$, $h\cdot f_3=0$ and $f_1\cdot f_3=0$,
Since $h=3\BNef[5]+2f_3+f_2$ where $\BNef[5]\cdot f_2=0$, $\BNef[5]\cdot f_3=1$ and $f_2\cdot f_3=1$,
 then $h$ satisfies the hypothesis of \cite[Proposition 5.7, ii)]{SD}
 thus $\varphi(X)$ is a cone over a rational normal twisted cubic in $\pp^3$.
 Moreover, by \cite[(5.9.2)]{SD} we have the description of $\varphi$ given in the statement.
 Observe that $E_3$ is mapped to the section $S$, $E_1$ and $E_4$ to the fiber of $\ff_3\to \pp^1$ through $p$ 
 and $E_2$ is contracted to the point $p$.

%To determine the degrees of the generators of the Cox ring, 
%we notice that by Theorem \ref{bl} the linear system associated 
%to a nef divisor $D\sim k\BNef[5]+f_3$, for $k \geq 2$, 
%%\claudia{debe decir: $D\sim k\BNef[5]+f_3$, for $k \geq 2$, ya que $\BNef[5]\cdot f_3=1$}
%is not base point free, and the only divisor in the Hilbert basis of the nef cone with this property is $\BNef[4]$.
%We apply all Tests excluding  the element $\BNef[4]$ from the sets $T_1,\cdots,T_5$.

By Theorem \ref{main-cox} the Cox ring $R(X)$ is generated in the following degrees: 
\[
f_1,\dots, f_{4},\ h_1,\ h_2,\ h_3,\ h_5,
\]
where $h_i:=\BNef[i]$ for $i=1,\dots,5$. Moreover, the set of such degrees is minimal, thus 
$R(X)$ has at least $8$ generators.
%Clearly $R(X)$ must contain the sections $s_1,\dots, s_{4}$ defining the $(-2)$-curves and a section in degree $h_5:=\BNef[5]$, since the elliptic fibration it defines has a unique reducible fiber. 
%\michela{what does the minimality test say?}
%By Proposition \ref{hyp}, the classes $h_i$ for $i=1,2,5$ are non-hyperelliptic and $h_3:=\BNef[3]$ is hyperelliptic. 
%Moreover  $h_1^2=h_2^2=12$, $h_3^2=6$ and $h_5^2=0$. 
\end{proof}

\begin{corollary}\label{uni8}
The moduli space of K3 surfaces $X$ with $\NS(X)\simeq U\oplus{A}_{2}$
is unirational.
\end{corollary}

\begin{proof}
By item i) of Proposition \ref{gen v8} we know that $X$ 
is a double cover of $\ff_3$ branched along a curve of type 
$S+B$, where $S^2=-3$ and $B\in |3S+10F|$ is smooth, where $F$ is a fiber of $\ff_3\to \pp^1$.
Moreover, $B$ is tangent at one point to the the fiber $F_p$ of $\ff_3\to \pp^1$ through $\{p\}=B\cap S$.
%The canonical divisor on $\FF_{3}$ is $K_{\FF_{3}}=-(2s+5f),$ thus
By the Riemann-Roch Theorem
\[
\chi(B)=\tfrac{1}{2}B(B-K_{\ff_3})+1=26.
\]
By \cite[Chapter V, Corollary 2.18]{RH} $B$ is very  ample,
thus $h^{2}(X,B)=0$, and by Serre duality:
$
h^{1}(X,B)=h^{1}(X,K_{\ff_3}-B).
$
The divisor $B-K_{\ff_3}\sim 5(S+3F)$ is not ample but by \cite[Chapter V, Corollary 2.18]{RH},
it contains an irreducible non-singular curve, in particular $B-K_{\ff_3}$
is nef. Since $(B-K_{\ff_3})^{2}>0$, by Mumford vanishing Theorem, we have
$h^{1}(X,K_{\ff_3}-B)=0$. The linear system $|3S+10F|$  is thus $25$ dimensional. 
Since $B$ is very ample, imposing a tangency condition on $B$ at a point gives a codimension
$1$ space. The automorphism group of $\ff_{3}$ is $8$ dimensional,
thus the quotient of the space of curves of type $S+B$ with $B\in |3S+10F|$ 
which are tangent to the fiber $F_p$ is a $16$ dimensional unirational variety.
Thus the moduli space of K3 surfaces $X$ with $\NS(X)\simeq U\oplus A_{2}$
is unirational.
\end{proof}

\subsection*{The family $\mathcal F_9$}

Let $X$ be a K3 surface with $\NS(X)\cong V_9=U(2)\oplus A_2$.
By Theorem \ref{main-eff} $X$ contains four classes of $(-2)$-curves,
whose intersection matrix is given by Table \ref{Intmat} and whose intersection graph is given 
in Figure \ref{fig5}.
The Hilbert basis of the nef cone contains seven classes, two of them defining 
elliptic fibrations $\BNef[6]$ and $\BNef[7]$ without sections and with one fiber of type $\tilde A_2$.  

\begin{figure}[H]
\begin{center}
\begin{tikzpicture}[scale=1]

\draw (0,0) -- (2,0);
\draw (1,0) -- (1,1);
\draw (0,0) -- (1,1);
\draw (2,0) -- (1,1);

\draw (0,0) node {$\bullet$};
\draw (2,0) node {$\bullet$};
\draw (1,0) node {$\bullet$}; 
\draw (1,1) node {$\bullet$}; 

\draw (1,1) node [above]{$E_{1}$};
\draw (0,0) node [below]{$E_{2}$};
\draw (1,0) node [below]{$E_{4}$};
\draw (2,0) node [below]{$E_{3}$};

\end{tikzpicture}
\end{center} 
\caption{Intersection graph of $(-2)$-curves of $\mathcal F_9$}\label{fig5}
\end{figure}
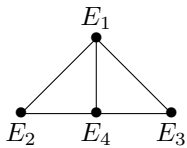

\begin{proposition}\label{gen v9}
Let $X$ be a K3 surface with $\NS(X)\cong V_9=U(2)\oplus A_2$. Then
\begin{enumerate}
\item  there is a minimal resolution $\varphi:X\to \overline X$ of a double cover 
$\pi:\overline X\to \pp^2$ branched along a plane sextic $B$ having two nodes $p_1,p_2$ and such that 
the line $L$ through the nodes is tangent to $B$ at one point;
\item the surface has four $(-2)$-curves:  two curves $R_1,R_2$ such that 
$\pi\varphi(R_i)=L$, $i=1,2$ and two curves $E_1,E_2$ with $\pi\varphi(E_i)=p_i$, $i=1,2$;
\item the Cox ring $R(X)$ is generated in the degrees given in Table \ref{TableGen rank4}, in particular 
has at least $10$ generators whose degrees are either classes of $(-2)$-curves 
or the classes $\BNef[i]$, for $i=1,2,3,4, 6, 7$, in the Hilbert basis of the nef cone.
%$T_1,\dots, T_4$ defining the $(-2)$-curves,
%$T_5,T_6$ defining two smooth fibers of the elliptic fibrations of $X$,
%$T_7,T_8$ whose degrees are elements of the Hilbert basis of the nef cone 
%with self-intersection $10$, $T_9,T_{10}$ whose degrees are elements of the Hilbert basis of the nef cone 
%with self-intersection $12$.
\end{enumerate}
\end{proposition}

\begin{proof}
Let $f_1,\dots, f_4$ be the classes of the $(-2)$-curves  and $h=\BNef[5]$.  
We have that  $h^2=2$,  $h\cdot f_i=1$ for $i=1,4$ and $h\cdot f_i=0$ for $i=2,3$.
By Corollary \ref{section} the associated linear system is base point free 
and thus defines a degree two morphism $\pi:X\to \pp^2$ which contracts the $(-2)$-curves 
of classes $f_2, f_3$. Since $f_2\cdot f_3=0$, the branch locus of $\pi$  
is a plane sextic $B$ with two nodes at $p,q\in \pp^2$. 
Moreover $f_1+f_2+f_3+f_4=h$ and $f_1\cdot f_4=1$. Thus the $(-2)$-curves 
of classes $f_1, f_4$ are mapped to a line $L$ passing through $p,q$ 
and tangent to $B$ at one more point.
The last item follows from Theorem \ref{main-cox}. 
%the Cox ring $R(X)$ is generated in the following degrees: 
%\[
%f_1,\dots, f_{4},
%\]
%\[ \BNef[1],\ \BNef[2],\ \BNef[3],\ \BNef[4],\ \BNef[6],\ \BNef[7],
%%2(\BNef[6] + \BNef[7]),
%\]
%where 
%$ \BNef[1], \BNef[2], \BNef[3], \BNef[4]$ are non-hyperelliptic by Proposition \ref{hyp}
%with $ \BNef[1]^2= \BNef[2]^2=12$, $ \BNef[3]^2= \BNef[4]^2=10$ and  $\BNef[6]^2=\BNef[7]^2=0$. 
%define the two elliptic fibrations of $X$, with $ \BNef[6] \cdot \BNef[7]=2$.
%Then, applying the tests (described in the proof of the Theorem \ref{main}) in these degrees we find that the degrees of the generators are those in the previous set excluding the special case.
%\michela{que dice el test de minimalidad en este caso?}
%\claudia{en este caso dice que todos los generadores son necesarios salvo $ 2(h_5+h_6)$, que es un caso excepcional}
%\michela{se pueden aplicar los tests a $2(h_5+h_6)$ para ver si hay generadores en ese grado o no}
%\claudia{se pueden aplicar, y da false en test1 y test2,  true en el test3, true en el test4, true test5, true test6}
%\claudia{$A:=2(h_5+h_6)$ is a nef divisor non-hyperelliptic with $A^2=16$ and $h^0(A)=10$}
%(the latter correspond to elliptic fibration and any minimal generating set of $R(X)$ must contain one generator for each elliptic fibration). So that,
 %we don't know how many generators there are without knowing generators relations of Cox ring.
 \end{proof}

\begin{proposition}\label{uni9}
The moduli space of K3 surfaces $X$ with $\NS(X)\simeq U(2)\oplus A_{2}$
is unirational. 
\end{proposition}

\begin{proof} By item i) of Proposition \ref{gen v9} $X$ is the minimal resolution 
of a double cover of $\pp^2$ branched along a plane sextic $B$ with two nodes $p_1,p_2$
such that the line $L$ through $p_1,p_2$ is tangent to $B$ at one more point $q$.
Up to projectivities we can assume that $p_{1}=(1,0,0)$, $p_{2}=(0,1,0)$ 
and $q=(-1,1,0)$. Using explicit computations, we get
that the projective linear system of sextics with nodes at $p_{1},p_{2}$
and tangent to the line $L$ at $q$ is $19$-dimensional. Modding
out by the automorphisms of $\pp^2$ preserving $p_{1},p_{2},q$, which are of
the form 
\[
\left(\begin{array}{ccc}
1 & 0 & *\\
0 & 1 & *\\
0 & 0 & *
\end{array}\right),
\]
we obtain an unirational space of dimension $16$. Since the dimension
of the moduli space of the K3 surfaces with $\NS(X)\simeq U(2)\oplus{A}_{2}$
is also $16$, the claims follows. 
\end{proof}

\subsection*{The family $\mathcal F_{10}$}

Let $X$ be a K3 surface with $\NS(X)\cong V_{10}=U(3)\oplus A_2$.
By Theorem \ref{main-eff} $X$ contains four $(-2)$-curves whose 
intersection matrix is given in Table \ref{Intmat} and whose intersection graph is described in Figure \ref{fig5}.
The Hilbert basis of the nef cone of $X$ contains five classes. 
The classes $\BNef[2], \BNef[3], \BNef[4]$  and $\BNef[5]$
define elliptic fibrations  without sections and with one 
fiber of type $\tilde A_2$. 

\begin{figure}[H]
\begin{center}
\begin{tikzpicture}[scale=1]

\draw (0,0) -- (1,0);
\draw (1,0) -- (1,1);
\draw (0,1) -- (1,1);
\draw (0,1) -- (0,0);
\draw (0,1) -- (1,0);
\draw (1,1) -- (0,0);

\draw (0,0) node {$\bullet$};
\draw (1,0) node {$\bullet$};
\draw (1,1) node {$\bullet$}; 
\draw (0,1) node {$\bullet$}; 

\draw (0,0) node [left]{$E_{4}$};
\draw (1,0) node [right]{$E_{3}$};
\draw (1,1) node [right]{$E_{2}$};
\draw (0,1) node [left]{$E_{1}$};

\end{tikzpicture}
\end{center} 
\caption{Intersection graph of $(-2)$-curves of $\mathcal F_{10}$}
\end{figure}
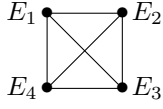

\begin{proposition}\label{gen v10}
Let $X$ be a K3 surface with $\NS(X)\cong V_{10}=U(3)\oplus A_2$. Then
\begin{enumerate}
\item  $X$ is isomorphic to a smooth quartic 
 surface in $\pp^3$ having one hyperplane section which is the union of four lines;
  \item a general $X$ can be defined by an equation of the form
 \[
F_0(x_0,\dots, x_3)G(x_0,\dots, x_3)+F_1(x_0,\dots, x_3)F_2(x_0,\dots, x_3)F_3(x_0,\dots, x_3)F_4(x_0,\dots, x_3)=0,
 \]
 where $F_i\in\cc[x_0,\dots, x_3]$ are homogeneous of degree one for $i=0,\dots, 4$
 and $G\in\cc[x_0,\dots, x_3]$ is homogeneous of degree three;
 \item the surface has four $(-2)$-curves: the four lines;
  \item the Cox ring of  $X$ is generated by $s_1,\dots, s_8$, where
  $s_1,\dots,s_4$ define the $(-2)$-curves and $s_5,\dots, s_8$ define
  each a smooth fiber of one of the elliptic fibrations of $X$;
%where $s_i$ be a generator of $H^0(f_i)$ for all i, and a basis $x_0, x_1, x_2\in H^0 (h)$. 
%where the degrees of the generators are given by the columns of the following matrix
%\[
%\left(
%\begin{array}{cccccccc}
%0& 0& -1& 0 & -1& -1 & -1& 0 \\
%0& -1 & 0 & 0 & -1 & -1 & 0 & -1 \\
%0& 1& 1& -1 & 1 & 2 & 0 & 0\\
%-1& 1 & 1 & 0& 2& 1& 0&0
%\end{array}
%\right).
%\]
\item  for  a very general $X$ as before we have an isomorphism 
\[
 \cc[T_1,\dots, T_8]/I\to R(X),\ T_i\mapsto s_i,
\]
where the degrees of the generators $T_i$ for $i=1,\dots, 8$ are given by the columns of the following matrix
%the Cox ring of a very general $X$ as before is isomorphic to $\cc[T_1,T_2,T_3,T_4,T_5,T_6,T_7,T_8]/I$, where the degrees of the generators are given by the columns of the following matrix
\[
\left(
\begin{array}{ccccccccc}
0&0 & -1 &0 & -1 &-1  &-1& 0 \\
0& -1& 0 &0 &-1  &-1  &0  &-1  \\
0& 1 &1  &-1 &1 &2 & 0 & 0 \\
-1&1  & 1 &0 &2 &1 & 0 & 0
\end{array}
\right)
\]
and the ideal $I$ is generated by the following polynomials:
\begin{eqnarray*}
&T_1T_2T_3T_4-\tilde F_0(T_1T_5,T_4T_6,T_2T_7,T_3T_8),\\
&T_5T_6T_7T_8-\tilde G(T_1T_5,T_4T_6,T_2T_7,T_3T_8),
\end{eqnarray*}
where $\tilde F_0=F\circ \varphi^{-1}$ and $\tilde G=G\circ \varphi^{-1}$ 
are obtained from $F_0,G$ in item ii) respectively composing with 
the coordinate change in $\pp^3$ given by 
\[
\varphi(x_0,x_1,x_2,x_3)=(F_1, F_2, F_3, F_4).
\]

\end{enumerate}
\end{proposition}

\begin{proof}
Let $f_1,\dots, f_4$ be the classes of the $(-2)$-curves and let 
$h=\BNef[1]$. Then  $h^2=4$ and $h\cdot f_i=1$ for all $i$.
Thus $h$ is ample and  is non-hyperelliptic by Proposition \ref{hyp}.
By Corollary \ref{section} the linear system associated to $h$ is base point free,
thus it defines an embedding of $X$ in $\pp^3$ as a smooth quartic surface.
Observe that $h=f_1+f_2+f_3+f_4$ with $h\cdot f_i=1$ for all $i$. 
This means that $X$ has one hyperplane section which is the union of 
four lines. 

By Theorem \ref{main-cox} the Cox ring $R(X)$ is generated in the following degrees: 
\[
f_1,\dots, f_{4},\ h,\  e_1,\ e_2,\ e_3,\ e_4,
\]
where $e_1, e_2, e_3, e_4$ define the four elliptic fibrations of $X$.
Clearly any minimal generating set of $R(X)$ must contain the sections $s_1,\dots,s_{4}$ defining 
the $(-2)$-curves and generators $s_5,\dots, s_8$ defining smooth fibers of  the elliptic fibrations. 
Moreover, observe that
%, $e_i:= \BNef[i+1]$ for $i=1,2,3,4$, such that
\[
h=e_1+f_1=e_2+f_4=e_3+f_2=e_4+f_3.
\]
An argument similar to the one in the proof of Proposition \ref{gen v7} 
shows that $X$ can be defined by an equation of the form 
$F_0G+F_1F_2F_3F_4=0$ in $\pp^3$, where $F_i$ are homogeneous of degree one for $i=0,\dots,4$ 
and $G$ of degree $3$. 
Since $F_1,\dots,F_4$ can be chosen to be independent, 
then $s_1 s_5, s_4s_6, s_2s_7, s_3s_8$ are a basis of $H^0(h)$.
Thus a generator in degree $h$ is not necessary.  
The first relation is due to the fact that the hyperplane section $F _0=0$ is the union of the
four lines whose defining sections are $T_i$ for $i=1,2,3,4$.
The second relation is due to the fact that the four elliptic curves defined by 
$s_i$, $i=5,\dots, 8$, are cut out by the cubic $G=0$.

It can be proved with the same type of argument used in the proof of 
\cite[Theorem 3.5]{A.C.L}  that the ideal $I$ is prime for general $G, F_0$. 
Since $\cc[T_1,\dots,T_8]/I$ is an integral domain of dimension $\dim R(X)= \dim (X)+ \rk \Cl(X)=6$ 
which surjects onto $R(X)$, then it is isomorphic to $R(X)$.
\end{proof}

The proof of the following result is similar to that of Corollary \ref{uni7}.

\begin{corollary}\label{uni10}
The moduli space of K3 surfaces $X$ with $\NS(X)\simeq U(3)\oplus {A}_{2}$
is unirational.
\end{corollary}

\begin{example}
Consider the following quartic in $\pp^{3}$:
\[
X:\ x_0x_1x_2x_3-(x_0+x_1+x_3)(x_0^{3}+x_0^{2}x_1+x_1^{3}+x_1x_2^{2}+x_2^{3}+x_0x_1x_3+x_2x_3^{2})=0.
\]
This is a smooth surface, with good reduction $X_{2}$ at prime $2$.
Using the Tate and Artin-Tate conjectures, one finds that $X_{2}$
has Picard number $4$ and 
\[
|\text{Br}(X_{2})|\cdot|\text{disc}(\NS(X_{2}))|=3^{3}.
\]
Since no K3 surface $Y$ with $\NS(Y)\simeq U\oplus{A}_{2}$ can
be embedded as a quartic in $\pp^{3}$ (this can be checked directly looking 
at the self-intersections of the elements in Hilbert basis of the nef cone of the family $\mathcal F_8$), 
we have $\NS(X_{2})\simeq U(2)\oplus{A}_{2}$.
Since $X$ has Picard number at least $4$, we conclude that $\NS(X)\simeq U(3)\oplus{A}_{2}$. 
\end{example}

\subsection*{The family $\mathcal F_{11}$}

Let $X$ be a K3 surface with $\NS(X)\cong V_{11}=U(6)\oplus A_2$.
By Theorem \ref{main-eff} $X$ contains six  $(-2)$-curves whose 
intersection matrix is given in Table \ref{Intmat}.
The Hilbert basis of the nef cone of $X$ contains $27$ classes, 
eight of them defining elliptic fibrations:
$\BNef[i]$ with $i=4, 7, 13, 16, 18, 21, 26, 27$ 
%$\BNef[4]$, $\BNef[7]$, $\BNef[13]$, $\BNef[16]$, $\BNef[18]$, $\BNef[21]$, $\BNef[26]$ and $\BNef[27]$ 
without sections and with one fiber of type $\tilde A_2$. 

\begin{proposition}\label{gen v11}
Let $X$ be a K3 surface with 
$\NS(X)\cong V_{11}=U(6)\oplus A_2$.
Then
\begin{enumerate}
\item  $X$ is isomorphic to a smooth quartic 
 surface in $\pp^3$ having three reducible hyperplane sections which are the union of two conics;
 \item $X$ contains six $(-2)$-curves: the six conics;
 \item the Cox ring of $X$ is generated in the degrees given in Table \ref{TableGen rank4}, in particular it 
 has at least $20$ generators.
 \end{enumerate}
\end{proposition}

\begin{proof}
Let $f_1,\dots, f_6$ be the classes of the $(-2)$-curves 
and $h=\BNef[25]$. Then $h^2=4$, $h\cdot f_i=2$ for all $i$ and by Proposition \ref{hyp} it is non-hyperelliptic.
By Corollary \ref{section} the associated linear system is base point free.
thus it defines an embedding of $X$ in $\pp^3$ as a smooth quartic surface.
Observe that  $h=f_1+f_{5}=f_2+f_4=f_3+f_{6}$.
This means that $X$ has three reducible hyperplane sections which are the union of two conics.

By Theorem \ref{main-cox} the Cox ring $R(X)$ is generated in the following degrees: 
\[
f_1,\dots, f_6,\ h_1,\dots, h_{14},\ h^*, \ h_1^*,\dots, h_{12}^*
\]
where $h_i$ and $h_i^*$ are classes in the Hilbert basis of the nef cone such that 
\[
h_i \in \{\BNef[j]:  j=4, 7, 9, 11, 13, 15, 16, 18, 20, 21, 23, 24, 26, 27\}
\]
\[
h_i^*\in \{\BNef[j]:  j=1-3, 5, 6, 8, 10, 12, 14, 17, 19, 22, 25\}.
\]
By Proposition \ref{hyp}, the classes $h_i$ and $h_i^*$ are non-hyperelliptic for all $i$, and if $v:=\BNef[i]$ we have 
\begin{enumerate}[]
\item $v^2=0$ for $i=4, 7, 13, 16, 18, 21, 26,27$
\item $v^2=6$ for $i=9, 11, 15, 20, 23, 24$
\item $v^2=10$ for $i=1, 2, 3, 5, 6, 8, 10, 12, 14, 17, 19, 22$.
\end{enumerate}
By the minimality test (Proposition \ref{minimal}) the degrees not marked with a star are necessary to generate $R(X)$.
Thus $R(X)$ has at least $20$ generators.
\end{proof}

Let $X$ be a K3 surface which has $3$ hyperplane sections $P_{k}$,
$k\in\{1,2,3\}$, each of which being union of two conics $C_{2k-1},C_{2k}$.
Let us denote by $L_{k}$ the line which is intersection of the planes
$P_{i},P_{j}$, for $\{k,i,j\}=\{1,2,3\}$. From the intersection
matrix, we see that on each line $L_{k}$ there are $4$ points $p_{4k-3},p_{4k-2},p_{4k-1},p_{4k}$
such that each point is the intersection of two conics. Moreover,
one can label the conics and the $12$ points, so that
\[
\begin{array}{ccc}
\{p_{1},p_{2},p_{11},p_{12}\}\subset C_{1}, & \{p_{2},p_{4},p_{7},p_{8}\}\subset C_{3}, & \{p_{6},p_{8},p_{10},p_{12}\}\subset C_{5}\\
\{p_{3},p_{4},p_{9},p_{10}\}\subset C_{2}, & \{p_{1},p_{3},p_{5},p_{6}\}\subset C_{4}, & \{p_{5},p_{7},p_{9},p_{11}\}\subset C_{6}.
\end{array}
\]
The situation is quite symmetric since for any choice $C\in\{C_{1},C_{2}\}$,
$C'\in\{C_{3},C_{4}\}$, $C''\in\{C_{5},C_{6}\}$, there are exactly
$9$ points among the $12$ points that are on the union of $C,C',C''$.
Moreover for any such choice and any line $L_{j}$ there are exactly
$3$ points among the $4$ points $p_{k}$ on $L_{j}$ which are on
the union of $C,C',C''$.

Conversely, in order to construct a K3 surface $X$ with the same
properties, let us consider three planes $P_{1},P_{2},P_{3}$ and
the lines $L_{k}=P_{i}\cap P_{j}$. Let us fix on line $L_{1}$ the
points $p_{1},p_{2},p_{4}$, on $L_{2}$ the points $p_{6},p_{7},p_{8}$,
on $L_{3}$ the points $p_{10},p_{11},p_{12}$. For a set $\{a,b,c,d\}$
among 
\[
\{1,2,11,12\},\,\,\{2,4,7,8\},\,\,\{6,8,10,12\},
\]
the (projective) linear system of quadrics passing through points
$p_{a},p_{b},p_{c},p_{d}$ is $5$ dimensional. Taking such a quadric
$Q_{k}$ and intersecting it with the hyperplane section $P_{k}$
containing points $p_{a},p_{b},p_{c},p_{d}$ gives a conic $C_{2k-1}$
containing the points $p_{a},p_{b},p_{c},p_{d}$. That conic is irreducible
if we choose $Q_{k}$ generic. The intersections of the conic $C_{2k-1}$
with the lines $L_{i},L_{j}$ ($\{i,j,k\}=\{1,2,3\}$) gives the remaining
points $p_{3},p_{5},p_{9}$. The projective linear system of quartics
containing the three conics $C_{1},C_{3},C_{5}$ we constructed is
$10$ dimensional. By choosing a generic quartic, we obtain a smooth
K3 surface $X$ containing conics $C_{1},C_{3},C_{5}$. We denote
by $C_{2},C_{4},C_{6}$ the conics contained in $X$ which are residual
to $C_{1},C_{3},C_{5}$. These conics are smooth since we supposed
$X$ is generic. The $6$ conics $C_{k}$ and the points $p_{1},\dots,p_{12}$
have the incidence relation we described before, thus the intersection
matrix of the curves $C_{1},\dots,C_{6}$ is the same as for the surfaces
with N\'eron-Severi group isomorphic to $U(6)\oplus{A}_{2}$.
From that construction, it is clear that:

\begin{corollary}\label{uni11}
The moduli of K3 surfaces with $\NS(X)\simeq U(6)\oplus {A}_{2}$
is unirational.
\end{corollary}

\subsection*{The family $\mathcal F_{12}$}

Let $X$ be a K3 surface with 
$$\NS(X)\cong V_{12}=\left[
\begin{array}{rr}
0&-3
\\
-3&2
\end{array}
\right]
\oplus A_2.$$
By Theorem \ref{main-eff} $X$ contains six classes of $(-2)$-curves 
whose intersection matrix is given in Table \ref{Intmat} and described 
in Figure \ref{fig6}.
The Hilbert basis of the nef cone of $X$ contains $33$ classes, 
two of them defining elliptic fibrations ($\BNef[30]$ and $\BNef[31]$) 
without sections and with one fiber of type $\tilde A_2$. 

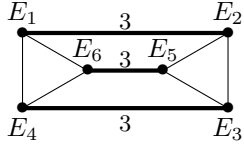
\begin{figure}[H]
\begin{center}
\begin{tikzpicture}[scale=1]

\draw (1,0) -- (1+0.866,0.5);
\draw (1,0) -- (1+0.866,-0.5);
\draw (0,0) -- (-0.866,0.5);
\draw (0,0) -- (-0.866,-0.5);

\draw (1+0.866,0.5) -- (1+0.866,-0.5);
\draw  (-0.866,0.5) -- (-0.866,-0.5);
\draw [ultra thick] (0,0) -- (1,0);
\draw [ultra thick] (-0.866,0.5) -- (1+0.866,0.5);
\draw [ultra thick] (-0.866,-0.5) -- (1+0.866,-0.5);

%[very thick] 

\draw (0,0) node {$\bullet$};
\draw (1,0) node {$\bullet$};
\draw (1+0.866,0.5) node {$\bullet$};
\draw (1+0.866,-0.5) node {$\bullet$};
\draw (-0.866,0.5) node {$\bullet$};
\draw (-0.866,-0.5) node {$\bullet$};

\draw (1,0) node [above]{$E_{5}$};
\draw (0,0) node [above]{$E_{6}$};
\draw (1+0.866,0.5) node [above]{$E_{2}$};
\draw (1+0.866,-0.5) node [below]{$E_{3}$};
\draw (-0.866,0.5) node [above]{$E_{1}$};
\draw (-0.866,-0.5) node [below]{$E_{4}$};

\draw (0.5,-0.07) node [above]{\small $3$};
\draw (0.5,0.43) node [above]{\small $3$};
\draw (0.5,-0.47) node [below]{\small $3$};

\end{tikzpicture}
\end{center} 
\caption{Intersection graph of $(-2)$-curves of $\mathcal F_{12}$}\label{fig6}
\end{figure}

\begin{proposition}\label{gen v12}
Let $X$ be a K3 surface with $\NS(X)\cong V_{12}$. 
Then 
\begin{enumerate}
\item  there is a double cover  $\pi:X\to \pp^2$ branched along a smooth plane sextic with 
three $3$-tangent lines $L_1, L_2, L_3$;
 \item $X$ can be defined by an equation of the following form in $\pp(1,1,1,3)$:
 \[
x_3^2=F_1(x_0,x_1, x_2)F_2(x_0,x_1, x_2)F_3(x_0,x_1, x_2)G_1(x_0,x_1, x_2)+G_2(x_0,x_1, x_2)^2,
 \]
 where $F_i\in\cc[x_0,x_1, x_2]$ are homogeneous of degree one for $i=1,2, 3$
 and $G_1,G_2\in\cc[x_0, x_1, x_2]$ are homogeneous of degree three;
\item the surface has six $(-2)$-curves: the curves $R_{ij}$, $i=1,2,3$, $j=1,2$, such that 
$\pi(R_{i1})=\pi(R_{i2})=L_i$;
 \item 
a minimal generating set of the Cox ring of $X$ is $s_1,\dots, s_8$, where: $s_1,\dots, s_6$ define the $(-2)$-curves and $s_7,s_8$ define smooth fibers of the two elliptic fibrations of $X$;
%where $s_i$ be a generator of $H^0(f_i)$ for all i, and a basis $x_0, x_1, x_2\in H^0 (h)$. 
%where the degrees of the generators are given by the columns of the following matrix
%\[
%\left(
%\begin{array}{cccccccc}
%1& 0 & -1& 0& 0 & -1 & -2& -1 \\
%-2& 0 & -3 & 1 & 0 & -2 & -3 & -3 \\
%1& 0 & 1& 0 & -1& 2 & 3 & 0 \\
%2& -1 & 1 & 0 & 0 & 1 & 3 & 0
%\end{array}
%\right).
%\]
\item  for  a very general $X$ as before we have an isomorphism 
\[
 \cc[T_1,\dots, T_8]/I\to R(X),\ T_i\mapsto s_i,
\]
where the degrees of the generators $T_i$ for $i=1, \dots, 8$ are given by the columns of the following matrix
%the Cox ring of a very general $X$ as before is isomorphic to $\cc[T_1,T_2,T_3,T_4,T_5,T_6,T_7,T_8]/I$, where the degrees of the generators are given by the columns of the following matrix
\[
\left(
\begin{array}{ccccccccc}
-1&0 & -1 &0 &0  &-1  &-2& -1 \\
-2& 0& -3 &1 &0 &-2 &-3&-3  \\
1& 0 &1  &0 &-1 &2 &3 &0 \\
2& -1 &1  &0 &0 &1 &3  & 0
\end{array}
\right)
\]
and the ideal $I$ is generated by the following polynomials:
\begin{eqnarray*}
&T_7T_8-\tilde G_1(T_1T_2,T_3T_4,T_5T_6),\\
&T_1T_3T_5T_7+T_2T_4T_6T_8-\tilde G_2(T_1T_2,T_3T_4,T_5T_6),
\end{eqnarray*}
where $\tilde G_i=G\circ \varphi^{-1}$, $i=1,2$ 
are obtained from $G_1,G_2$   in item ii) respectively composing with 
the coordinate change in $\pp^2$ given by 
\[
\varphi(x_0,x_1,x_2)=(F_1, F_2, F_3).
\]

\end{enumerate}
\end{proposition}

\begin{proof}
Let $f_1,\dots,f_6$ be the classes of the $(-2)$-curves 
and $h=\BNef[33]$. Then  $h^2=2$,  $h\cdot f_i=1$ for all $i$.
Thus $h$ is ample. By Corollary \ref{section} the associated linear system is base point free,
thus it defines a double cover $\pi:X\to \pp^2$ branched along a smooth plane sextic $B$.
Since $h=f_1+f_{2}=f_3+f_{4}=f_5+f_6,$ 
the image by $\pi$ of the six $(-2)$-curves of $X$ are three lines $L_1,L_2,L_3 \subseteq \pp^2$ 
such that $\pi^{-1}(L_i)$ is the union of two smooth rational curves for each $i=1,\dots, 3$.
This implies that the lines $L_i$ are $3$-tangent to $B$.
%Observe that $\BNef[30]+\BNef[31]=3h$, thus there exists a cubic 

By Theorem \ref{main-cox} the Cox ring $R(X)$ is generated in the following degrees: 
\[
f_1,\dots, f_{6},\ h^*,\ e_1, e_2 
\]
where $e_1$ and $e_2 $ are the two elliptic fibrations of $X$.
Clearly any minimal generating set of $R(X)$ must contain the sections $s_1,\dots,s_{6}$ defining 
the $(-2)$-curves of $X$ and one section defining a smooth fiber for each elliptic fibration.

It can be proved as in the proof of Proposition \ref{gen v6} that $X$  is a double cover of $\pp^2$ branched along a smooth plane sextic $B$ defined by an equation of the form $F_1F_2F_3G_1+G_2^2=0$,
where $F_1,F_2,F_3\in \cc[x_0,x_1,x_2]$ are homogeneous of degree one and $G_1,G_2\in \cc[x_0,x_1,x_2]$ of degree $3$.
Since $F_1,F_2,F_3$ can be chosen to be independent, then  
$s_1s_2, s_3s_4, s_5s_6$ give a basis of  $H^0(h)$. 
Thus a generator of $R(X)$ in degree $h$ is not necessary. 

The first element in $I$ is due to the fact that the preimage by $\pi$ of the plane cubic $G_1=0$ 
is the union of two smooth elliptic curves which are fibers of the two distinct elliptic fibrations of $X$ 
(in fact $\BNef[30]+\BNef[31]=3h$).
The last relation follows from the fact that 
\[
(x_3+G_2 )(x_3-G_2)=F_1F_2F_3G_1,
\] 
thus up to renumbering we can assume $x_3+G_2=s_1s_3s_5s_7$, $x_3-G_2=s_2s_4s_6s_8$, so that $2G_2=s_1s_3s_3s_5s_7 -s_2s_4s_6s_8$. Up to rescaling the generators $s_i$ we obtain the last relation.
 
t can be proved with the same type of argument used in the proof of 
\cite[Theorem 3.5]{A.C.L} that the ideal $I$ is prime for general $G_1, G_2$. 
Thus $\cc[T_1,\dots,T_8]/I\cong R(X)$,  since it is an integral domain of dimension $\dim R(X)= \dim (X)+ \rk \Cl(X)=6$.
%There are relations of the following form:
%\[
%s_7s_8-g_1(s_1s_2,s_3s_4,s_5s_6),\,\,s_1s_2s_3s_4s_5s_6-g_2(s_1s_2,s_3s_4,s_5s_6).
%\]
%where the relation comes from the equation of the cubic containing $X$ observing that
%we can assume $\ell_1=s_1s_2, \ell_2=s_3s_4, \ell_3=s_5s_6$ and $=s_1s_2s_3$.
%
%The relations are due to the fact that $s_1+s_2+s_3+s_4+s_5+s_6=3h$ and $s_7+s_8=3h$,  
%
%
%because the hyperplane section is the union of 
%four lines, that correspond to the four $(-2)$-curves. (falta justificar la segunda relacion, que viene de escribir todas las combinaciones en magma).
\end{proof}

The proof of the following result is similar to the one of Corollary \ref{uni6}.

\begin{corollary}\label{uni12}
The moduli space of K3 surfaces with $\NS(X)\cong V_{12}$ is unirational.
\end{corollary}

\subsection*{The family $\mathcal F_{13}$}

Let $X$ be a K3 surface with 
$$\NS(X)\cong V_{13}=\left[
\begin{array}{rrrr}
2&-1&-1&-1
\\
-1&-2&0&0
\\
-1&0&-2&0
\\
-1&0&0&-2
\end{array}\right].$$
By Theorem \ref{main-eff} $X$ contains six  $(-2)$-curves 
whose intersection matrix is given in Table \ref{Intmat}
and described in Figure \ref{fig7}.
The Hilbert basis of the nef cone of $X$ contains $39$ classes of positive self-intersection. 
Thus $X$ has no elliptic fibrations. 

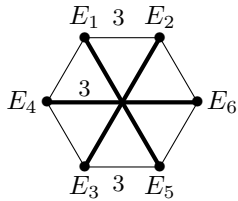
\begin{figure}[H]
\begin{center}
\begin{tikzpicture}[scale=1]
%[very thick]

\draw  (1,0) -- (0.5,0.86);
\draw  (-1,0) -- (-0.5,0.86);
\draw  (-0.5,-0.86) -- (0.5,-0.86);
\draw  (-0.5,0.86) -- (0.5,0.86);
\draw  (-0.5,-0.86) -- (-1,0);
\draw  (0.5,-0.86) -- (1,0);

\draw [ultra thick] (0,0) -- (0.5,-0.86);
\draw [ultra thick] (0,0) -- (-0.5,-0.86);
\draw [ultra thick] (0,0) -- (-0.5,0.86);
\draw [ultra thick] (-1,0) -- (1,0);
\draw [ultra thick] (0,0) -- (0.5,0.86);

\draw (1,0) node {$\bullet$};
\draw (0.5,0.86) node {$\bullet$};
\draw (-1,0) node {$\bullet$};
\draw (-0.5,0.86) node {$\bullet$};
\draw (-0.5,-0.86) node {$\bullet$};
\draw (0.5,-0.86)  node {$\bullet$};

\draw (1,0) node  [right]{$E_{6}$};
\draw (0.5,0.86) node [above]{$E_{2}$};
\draw (-1,0) node  [left]{$E_{4}$};
\draw (-0.5,0.86) node  [above]{$E_{1}$};
\draw (-0.5,-0.86) node  [below]{$E_{3}$};
\draw (0.5,-0.86) node   [below]{$E_{5}$};

\draw (-0.5,-0.04) node [above]{\small $3$};
\draw (-0.05,0.9) node [above]{\small $3$};
\draw (-0.05,-1.3) node [above]{\small $3$};

\end{tikzpicture}
\end{center}
\caption{Intersection graph of $(-2)$-curves of $\mathcal F_{13}$}\label{fig7} 
\end{figure}

\begin{proposition}\label{gen v13}
Let $X$ be a K3 surface with $\NS(X)\cong V_{13}$. Then
\begin{enumerate}
\item   there is a double cover $\pi:X\to \pp^2$ branched along 
 a smooth plane sextic with three $3$-tangent lines $L_1,L_2$ and $L_3$;
\item the surface has six $(-2)$-curves:  the six curves $R_{ij}$, $i=1,2,3$, $j=1,2$ such that 
$\pi(R_{i1})=\pi(R_{i2})=L_i$;
\item the Cox ring is generated in the degrees given in Table \ref{TableGen rank4}, in particular has at least $24$ generators.
 \end{enumerate}
\end{proposition}

\begin{proof}
Let $f_1,\dots, f_6$ be the classes of the $(-2)$-curves and $h=\BNef[1]$.
Then $h^2=2$ and $h\cdot f_i=1$ for all $i$.
 Thus $h$ is ample and the associated linear system is base point free by Corollary \ref{section}.
 Thus it defines a double cover $\pi:X\to\pp^2$ branched along a smooth plane sextic $B$.
Since $h=f_3+f_{5}=f_1+f_{2}=f_4+f_6,$  the image by $\pi$ of the six $(-2)$-curves of $X$ 
are three lines $L_1,L_2,L_3\subseteq \pp^2$ 
such that $\pi^{-1}(L_i)$ is the union of two smooth rational curves for each $i=1,2,3$.
This implies that $L_1,L_2,L_3$ are $3$-tangent to $B$.

By Theorem \ref{main-cox} the Cox ring $R(X)$ is generated in the following degrees: 
\[
f_1,\dots, f_{6},\ h^*,\ h_1,h_2,\dots, h_{18},
\]
where $h_1,\ h_2,\dots,\ h_{18}$  are classes in the Hilbert basis of the nef cone,
non-hyperelliptic, 
with self-intersection $4$ (six of them), $26$ (six of them) and $28$ (six of them).
By the minimality test (Proposition \ref{minimal}) $R(X)$ has a generator in all the above degrees,
except possibly for $h$.
%Clearly any minimal generating set of $R(X)$ must contain the sections $s_1,\dots, s_{6}$ defining 
%the $(-2)$-curves of $X$, a basis $s_3s_5,\ s_1s_2,\ s_4s_6$ of $H^0(h)$.
%By Proposition \ref{hyp}, the classes $h_i$ for $i=1,\dots,18$ are non-hyperelliptic and if $v:=\BNef[i]$ we have to
%\begin{enumerate}
%\item $v^2=4$ for $i=7,\ 9,\ 12,\ 14,\ 16,\ 17$
%\item $v^2=26$ for $i=27,\ 28,\ 30,\ 32,\ 34,\ 35$
%\item $v^2=28$ for $i=31,\ 33,\ 36,\ 37,\ 38,\ 39$
%\end{enumerate}
%and we don't know how many generators there are without knowing generators relations of Cox ring.
%Therefore, the Cox ring of a very general $X$ has at least six generators.
\end{proof}

\begin{remark}
In \cite{XR2} the author proved that a K3 surface which is the double
cover of the plane branched over a generic sextic which has $3$ tritangent
lines has N\'eron-Severi group as above, or isomorphic to  the lattice $V_{12}$. 
%\begin{figure}[h]
%\label{fig:Image-of-theTriangle}
%\caption{Image of }
%\end{figure}\\
So one can see how differently the configurations of the $3$ tritangent
lines lift to the K3 surface. 
\end{remark}

\subsection*{The family $\mathcal F_{14}$}

Let $X$ be a K3 surface with 
\[
\NS(X)\cong V_{14}=\left[
\begin{array}{rrrr}
12&-2&0&0
\\
-2&-2&-1&0
\\
0&-1&-2&-1
\\
0&0&-1&-2
\end{array}
\right]
.\]
By Theorem \ref{main-eff} $X$ contains eight classes of  $(-2)$-curves 
whose intersection matrix is given in Table \ref{Intmat}.
The Hilbert basis of the nef cone of $X$ contains $111$ classes, none of them defining elliptic fibrations. 

\begin{proposition}\label{gen v14}
Let $X$ be a K3 surface with $\NS(X)\cong  V_{14}$. Then
\begin{enumerate}
\item  there is a double cover  $\pi:X\to \pp^2$ branched along a smooth plane sextic with one $3$-tangent line $L$ 
and three $6$-tangent conics $C_1,C_2$ and $C_3$;
\item the surface has eight $(-2)$-curves: the curves $R_{ij}$, $i=1,2,3$ such that 
$\pi(R_{i1})=\pi(R_{i2})=C_i$ and the curves $S_1,S_2$ such that 
$\pi(S_1)=\pi(S_2)=L$;
\item the Cox ring of $X$ has at least $71$ generators whose degrees are either classes of 
$(-2)$-curves or  elements of the Hilbert basis of the nef cone.
\end{enumerate}
\end{proposition}

\begin{proof}
Let $f_1,\dots, f_8$ be the classes of the $(-2)$-curves 
and $h=\BNef[8]$. Then  $h^2=2$,  $h\cdot f_i=2$ for $i=1,\dots, 5,8$ and $h\cdot f_i=1$ for $i=6,7$.
By Corollary \ref{section} the associated linear system  is base point free, thus defines 
a double cover of $\pp^2$ branched along a smooth plane sextic $B$.
Since $h=f_6+f_{7},$ and $2h=f_4+f_{5}=f_2+f_{8}=f_1+f_3,$ 
the image by $\pi$ of the eight $(-2)$-curves of $X$ are three smooth conics $C_1,C_2,C_3\subseteq \pp^2$
 and one line $L\subseteq \pp^2$,
such that $\pi^{-1}(C_i)$, $i=1,2,3$ and  $\pi^{-1}(L)$ is the union of two smooth rational curves.
The last item follows from Theorem \ref{main-cox}.
\end{proof}

\section{Magma code}\label{magma}
We mention here the Magma \cite{B.C.P}  libraries which are used in the proof of Theorem \ref{main-cox}. 
Section 3.3 of \cite{A.C.L} contains a detailed the description of the functions in each library.
\\

\begin{enumerate}[$\bullet$]
\item \texttt{LSK3Lib.m}: to deal with linear systems on K3 surfaces,
\item\texttt{Find-2.m}:  to compute the set of $(-2)$-curves of a Mori dream K3 
surface,
\item \texttt{TestLib.m}: contains functions based on the results in section \ref{cox},
\item \texttt{MinimalLib.m}:  contains functions which check the minimality of a generating set of $R(X)$.\\
\end{enumerate}

Moreover, we will add to the arXiv version of the present paper the following two files:\\

\begin{enumerate}[$\bullet$]
\item \texttt{K3Rank4}: text file containing the intersection matrix and the list of classes of $(-2)$-curves for all Mori dream K3 surfaces of Picard number $3$
\item \texttt{Gen(K3Rank4)}: text file containing the list of classes which pass all tests 
in  \texttt{TestLib.m} (thus contains the degrees of a generating set of $R(X)$), for all Mori dream K3 surfaces of Picard number $4$).
\end{enumerate}

%\michela{to be completed}

 \section{Tables}\label{tables}
% This section contains the tables describing the effective and nef cone (Table 1), the intersection matrix of $(-2)$- curves (Table 4)
% and the degrees of a generating set of $R(X)$ (Table 6) for
% Mori dream K3 surfaces of Picard number four.
% We recall that $\Cl(X)$ denotes the Picard lattice of the surface,
% $E(X)$ is the set of generators of the extremal rays of the effective cone 
% (i.e. the set of classes of the $(-2)$-curves), 
% $\BEff(X)$ is the Hilbert  basis 
% of the effective cone, $N(X)$ is the set of generators of the 
% extremal rays of the nef cone and 
% $\BNef(X)$ is the Hilbert basis of the nef cone.
% 
% //////////////////////

This section contains the tables with the relevant information about Mori dream K3 surfaces of Picard number four: 
N\'eron-Severi lattice, effective cone and its Hilbert basis, nef cone and its Hilbert basis. 
Moreover, we provide a nef and big divisor in the Hilbert basis of the nef cone with minimum self-intersection in each family of K3 surfaces 
and its intersection properties with $(-2)$-curves. Finally, we will give the degrees of a set of generators of the Cox ring $R(X)$.

We recall that in the tables we will adopt this notation: $\NS(X)$ denotes the N\'eron-Severi lattice of the surface, $\Eff(X)$ is the effective cone and $\BEff(X)$ is its Hilbert basis, 
$E(X)$ is the set of generators of the extremal rays of the effective cone (i.e. the set of classes of the $(-2)$-curves),  $\Nef(X)$ is the nef cone and $\BNef(X)$ is its Hilbert basis, 
$N(X)$ is the set of generators of the extremal rays of $\Nef(X)$.

In Table \ref{Table1} we give $E(X)$, $\BEff(X)$, $N(X)$ and $\BNef(X)$, in Table \ref{Intmat} we give the intersection matrix of $(-2)$-curves for each family of Mori dream K3 surfaces of Picard number four.
 In Tables \ref{Table2} and \ref{Table3} we give the Hilbert basis of the nef cone of $X$ when the lattice $\NS(X)$ is isometric to $V_1$, $V_{2}$, $V_{13}$ or $V_{14}$. 

For each family of Mori dream K3 surfaces of Picard number three, in Table \ref{Table5} we give a nef and big class $H\in \BNef(X)$ of minimal self-intersection and its intersection properties with the $(-2)$-curves.

In Table \ref{TableGen rank4} we give the degrees of a set of generators of the Cox ring $R(X)$ when $\NS(X)$ is not isometric to $V_{14}$.
All degrees in the Table \ref{TableGen rank4} are necessary to generate $R(X)$, except possibly for those marked with a star.
In case $\NS(X)\cong V_{14}$ we give a subset of the degrees of a minimal generating set of $R(X)$.

\newpage

\begin{table}[H] 
\centering\renewcommand{\arraystretch}{1.3}\setlength{\tabcolsep}{2pt}
\caption{Effective and Nef cones.}
\label{Table1}
\vspace{.2cm}
%\footnotesize
\begin{tabular}{|c|c|c|c|c|c|c|}
\hline
 %\multicolumn{5}{ |c| }{ K3 surfaces without elliptic curves} \\
 %\hline
${N}^\circ$&$\NS(X)$& $E(X)$& $\BEff(X)$ & $N(X)$ &  $\BNef(X)$ \\
\hline\hline
1& $V_1$ & 
 $\begin{array}{c}
(1, -1,  0, -2),\\
    (1, -1,  2,  0),\\
    (0,  0, -1,  0),\\
    (2, -2,  2, -3),\\
    (0,  1,  0,  0),\\
    (2, -2,  3, -2),\\
    (1,  0,  1, -2),\\
    (1, -2,  0, -1),\\
    (1,  0,  2, -1),\\
    (1, -2,  1,  0),\\
    (2, -3,  2, -2),\\
    (0,  0,  0,  1)
\end{array}$  & 
 $\begin{array}{c}
 E(X)\\
 \cup\\
 \{(1,-1,1,-1)\} 
 \end{array}$ & 
$\begin{array}{c}
(1, -2,  0,  0),\\
    (1,  0,  0, -2),\\
    (1,  0,  0,  0),\\
    (1,  0,  2,  0),\\
    (3, -4,  0, -4),\\
    (3, -4,  2, -4),\\
    (3, -4,  4, -2),\\
    (3, -4,  4,  0),\\
    (3, -2,  4, -4),\\
    (3,  0,  4, -4),\\
    (5, -8,  4, -4),\\
    (5, -4,  4, -8),\\
    (5, -4,  8, -4),\\
    (7, -8,  8, -8)
\end{array}$& 
See Table \ref{Table2}\\
\hline
2&$V_2$ & 
$\begin{array}{c}
( 1,  1,  1,  0),\\
    ( 0,  0,  0,  1),\\
    (-1,  1,  1,  0),\\
    ( 0,  0, -1, -1),\\
    ( 0,  1,  1, -1),\\
    ( 0,  1,  2,  1)
\end{array}$ &
$E(X)$ & 
 $\begin{array}{c}
(-3,  6,  4, -4),\\
    (-3,  6,  8,  4),\\
    (-1,  1,  0,  0),\\
    (-1,  3,  4,  0),\\
    ( 1,  1,  0,  0),\\
    ( 1,  3,  4,  0),\\
    ( 3,  6,  4, -4),\\
    ( 3,  6,  8,  4)
\end{array}$& See Table \ref{Table2}\\
\hline
3&$V_3$ & $\begin{array}{c}
 (1,  0,  2,  1),\\
    (0,  0,  0, -1),\\
    (1,  1,  2,  2),\\
    (0,  1,  0,  0),\\
    (0, -1, -1,  0)
\end{array}$  &
$E(X)$ &
$\begin{array}{c}
(1,  0,  0,  0),\\
    (1,  1,  2,  1),\\
    (2, -1,  2,  1),\\
    (2,  1,  2,  3),\\
    (3,  0,  4,  4)
\end{array}$ & $\begin{array}{c}
(1, 0, 0, 0),
    (1, 0, 1, 1),\\
    (1, 1, 2, 1),
    (2, -1, 2, 1),\\
    (2, 0, 2, 1),
    (2, 1, 2, 2),\\
    (2, 1, 2, 3),
    (3, 0, 4, 3),\\
    (3, 0, 4, 4),
    (3, 1, 4, 4)
\end{array}$ \\
\hline
4&$V_4$ & $\begin{array}{c}
( 0,  0,  0, -1),\\
    (-1,  0,  1,  0),\\
    ( 1, -1,  0,  0),\\
    (-1,  0,  0,  1),\\
    ( 0,  0, -1,  0)
\end{array}$  &
$E(X)$ &
$\begin{array}{c}
(-2, -2,  0,  1),\\
    (-2, -2,  1,  0),\\
    (-2, -2,  1,  1),\\
    (-1, -1,  0,  0),\\
    (-1,  0,  0,  0)
    \end{array}$& $\begin{array}{c}
(-2, -2, 0, 1),\\
    (-2, -2, 1, 0),\\
    (-2, -2, 1, 1),\\
    (-1, -1, 0, 0),\\
    (-1, 0, 0, 0)
\end{array}$\\
\hline
5&$V_5$ & $\begin{array}{c}
 ( 0, -1,  0,  1),\\
    ( 0, -1,  1,  0),\\
    ( 0,  0,  0, -1),\\
    (-1,  0,  1,  0),\\
    (-1,  0,  0,  1),\\
    ( 0,  0, -1,  0)
\end{array}$ &
$E(X)$ &
$\begin{array}{c}
(-1, -1,  0,  1),\\
    (-1, -1,  1,  0),\\
    (-1, -1,  1,  1),\\
    (-1,  0,  0,  0),\\
    ( 0, -1,  0,  0)
\end{array}$ & $\begin{array}{c}
(-1, -1, 0, 1),\\
    (-1, -1, 1, 0),\\
    (-1, -1, 1, 1),\\
    (-1, 0, 0, 0),\\
    (0, -1, 0, 0)
\end{array}$ \\
\hline
\end{tabular}
\end{table}

\begin{table}[H] 
\centering\renewcommand{\arraystretch}{1.3}\setlength{\tabcolsep}{2pt}
%\caption{Effective and Nef cone for K3 surfaces with $\varrho(X)=4$.}
%\label{Eff and Nef rank4 1}
\vspace{1cm}
%\footnotesize
\begin{tabular}{|c|c|c|c|c|c|c|}
\hline
 %\multicolumn{5}{ |c| }{ K3 surfaces without elliptic curves} \\
 %\hline
${N}^\circ$&$\NS(X)$&$E(X)$&$\BEff(X)$ & $N(X)$&$\BNef(X)$ \\
\hline\hline
6& $V_6$ & 
 $\begin{array}{c}
( 0,  0,  0,  1),\\
    (-2, -2, -3, -2),\\
    ( 0, -1,  0, -1),\\
    ( 0,  0,  1,  0),\\
    (-1,  0,  0, -1),\\
    (-2, -2, -2, -3),\\
    ( 0, -1, -1,  0),\\
    (-1,  0, -1,  0)
\end{array}$  & 
 $ E(X)$ & 
$\begin{array}{c}
(-4, -4, -6, -3),\\
    (-4, -4, -3, -6),\\
    (-3, -2, -3, -3),\\
    (-2, -3, -3, -3),\\
    (-2, -2, -3,  0),\\
    (-2, -2,  0, -3),\\
    (-1,  0,  0,  0),\\
    ( 0, -1,  0,  0)
\end{array}$& 
$\begin{array}{c}
 (-4, -4, -6, -3),
    (-4, -4, -3, -6),\\
    (-3, -3, -4, -3),
    (-3, -3, -3, -4),\\
    (-3, -2, -3, -3),
    (-3, -2, -3, -2),\\
    (-3, -2, -2, -3),
    (-2, -3, -3, -3),\\
    (-2, -3, -3, -2),
    (-2, -3, -2, -3),\\
    (-2, -2, -3, -1),
    (-2, -2, -3, 0),\\
    (-2, -2, -1, -3),
    (-2, -2, 0, -3),\\
    (-1, -1, -1, -1),
    (-1, -1, -1, 0),\\
    (-1, -1, 0, -1),
    (-1, 0, 0, 0),\\
    (0, -1, 0, 0)
    %son 19
\end{array}$\\
\hline
7&$V_7$ & 
$\begin{array}{c}
 ( 0, -1,  0,  1),\\	
    ( 0,  0,  0, -1),\\
    (-1, -1, -2,  1),\\
    (-1, -1, -1,  2),\\
    ( 0,  0,  1,  0),\\
    (-1,  0,  0,  1),\\
    ( 0, -1, -1,  0),\\
    (-1,  0, -1,  0)
\end{array}$ &
$E(X)$ & 
 $\begin{array}{c}
 (-2, -1, -2,  2),\\
    (-1, -2, -2,  2),\\
    (-1, -1, -2,  0),\\
    (-1, -1,  0,  2),\\
    (-1,  0,  0,  0),\\
    ( 0, -1,  0,  0)
\end{array}$& $\begin{array}{c}
(-2, -2, -3, 2),
    (-2, -2, -2, 3),\\
    (-2, -1, -2, 1),
    (-2, -1, -2, 2),\\
    (-2, -1, -1, 2),
    (-1, -2, -2, 1),\\
    (-1, -2, -2, 2),
    (-1, -2, -1, 2),\\
    (-1, -1, -2, 0),
    (-1, -1, -1, 0),\\
    (-1, -1, -1, 1),
    (-1, -1, 0, 1),\\
    (-1, -1, 0, 2),
    (-1, 0, 0, 0),\\
    (0, -1, 0, 0)
    %son 15
\end{array}$\\
\hline
8&$V_8$ & $\begin{array}{c}
 ( 0,  0,  0,  1),\\
    (-1,  0,  1,  0),\\
    ( 1, -1,  0,  0),\\
    ( 0,  0, -1, -1)
\end{array}$  &
$E(X)$ &
$\begin{array}{c}
(-3, -3,  1, -1),\\
    (-3, -3,  2,  1),\\
    (-1, -1,  0,  0),\\
    (-1,  0,  0,  0)
\end{array}$ & $\begin{array}{c}
(-3, -3, 1, -1),
    (-3, -3, 2, 1),\\
    (-2, -2, 1, 0),
    (-1, -1, 0, 0),\\
    (-1, 0, 0, 0)
\end{array}$ \\
\hline
9&$V_9$ & $\begin{array}{c}
( 0,  0,  0,  1),\\
    ( 0, -1,  1,  0),\\
    (-1,  0,  1,  0),\\
    ( 0,  0, -1, -1)
\end{array}$  &
$E(X)$ &
$\begin{array}{c}
(-3, -3,  2, -2),\\
    (-3, -3,  4,  2),\\
    (-1,  0,  0,  0),\\
    ( 0, -1,  0,  0)
    \end{array}$& $\begin{array}{c}
(-3, -3, 2, -2),
    (-3, -3, 4, 2),\\
    (-2, -2, 1, -1),
    (-2, -2, 2, 1),\\
    (-1, -1, 1, 0),
    (-1, 0, 0, 0),\\
    (0, -1, 0, 0)
\end{array}$\\
\hline
10&$V_{10}$ & $\begin{array}{c}
 ( 0,  0,  0, -1),\\
    ( 0, -1,  1,  1),\\
    (-1,  0,  1,  1),\\
    ( 0,  0, -1,  0)
 \end{array}$ &
$E(X)$ &
$\begin{array}{c}
 (-1, -1,  1,  2),\\
    (-1, -1,  2,  1),\\
    (-1,  0,  0,  0),\\
    ( 0, -1,  0,  0)
\end{array}$ & $\begin{array}{c}
 (-1, -1, 1, 1),
    (-1, -1, 1, 2),\\
    (-1, -1, 2, 1),
    (-1, 0, 0, 0),\\
    (0, -1, 0, 0)
\end{array}$ \\
\hline
\end{tabular}
\end{table}

\begin{table}[H] 
\centering\renewcommand{\arraystretch}{1.3}\setlength{\tabcolsep}{2pt}
%\caption{Effective and Nef cone for K3 surfaces of $\varrho(X)=4$.}
%\label{Eff and Nef rank4 4}
\vspace{1cm}
%\footnotesize
\begin{tabular}{|c|c|c|c|c|c|c|}
\hline
 %\multicolumn{5}{ |c| }{ K3 surfaces without elliptic curves} \\
 %\hline
${N}^\circ$&$\NS(X)$& $E(X)$& $\BEff(X)$ & $N(X)$ &  $\BNef(X)$ \\
\hline\hline
11& $V_{11}$ & 
 $\begin{array}{c}
( 0,  0,  0, -1),\\
    ( 0, -1,  1,  1),\\
    (-1, -1,  3,  2),\\
    (-1,  0,  1,  1),\\
    (-1, -1,  2,  3),\\
    ( 0,  0, -1,  0)
\end{array}$  & 
 $ E(X)$ & 
$\begin{array}{c}
(-3, -2,  6,  6),\\
    (-2, -3,  6,  6),\\
    (-2, -1,  2,  4),\\
    (-2, -1,  4,  2),\\
    (-1, -2,  2,  4),\\
    (-1, -2,  4,  2),\\
    (-1,  0,  0,  0),\\
    ( 0, -1,  0,  0)
\end{array}$& 
$\begin{array}{c}
 (-3, -3, 7, 7),
    (-3, -2, 5, 6),\\
    (-3, -2, 6, 5),
    (-3, -2, 6, 6),\\
    (-2, -3, 5, 6),
    (-2, -3, 6, 5),\\
    (-2, -3, 6, 6),
    (-2, -2, 3, 5),\\
    (-2, -2, 4, 5),
    (-2, -2, 5, 3),\\
    (-2, -2, 5, 4),
    (-2, -1, 2, 3),\\
    (-2, -1, 2, 4),
    (-2, -1, 3, 2),\\
    (-2, -1, 3, 3),
    (-2, -1, 4, 2),\\
    (-1, -2, 2, 3),
    (-1, -2, 2, 4),\\
    (-1, -2, 3, 2),
    (-1, -2, 3, 3),\\
    (-1, -2, 4, 2),
    (-1, -1, 1, 1),\\
    (-1, -1, 1, 2),
    (-1, -1, 2, 1),\\
    (-1, -1, 2, 2),
    (-1, 0, 0, 0),\\
    (0, -1, 0, 0)
    %son 27
\end{array}$\\
\hline
12&$V_{12}$ & 
 $\begin{array}{c}
 (-1,  -2,  1, 2),\\
    (0,  0,  0,  -1),\\
    (-1,  -3, 1,  1),\\
    (0,  1,  0,  0),\\
    (0,  0,  -1,  0),\\
    (-1, -2,  2,  1)
\end{array}$ &
$E(X)$ & 
 $\begin{array}{c}
(-7, -15,   9,   9),\\
    (-5, -12,   3,   6),\\
    (-5, -12,   6,   3),\\
    (-4,  -6,   3,   6),\\
    (-4,  -6,   6,   3),\\
    (-2,  -3,   0,   0),\\
    (-2,  -3,   3,   3),\\
    (-1,  -3,   0,   0)
    \end{array}$& 
     $\begin{array}{c}
    (-7, -15, 9, 9),
    (-5, -12, 3, 6),\\
    (-5, -12, 6, 3),
    (-5, -11, 6, 6),\\
    (-4, -9, 3, 5),
    (-4, -9, 4, 5),\\
    (-4, -9, 5, 3),
    (-4, -9, 5, 4),\\
    (-4, -6, 3, 6),
    (-4, -6, 6, 3),\\
    (-3, -7, 3, 3),
    (-3, -6, 2, 4),\\
    (-3, -6, 3, 4),
    (-3, -6, 4, 2),\\
    (-3, -6, 4, 3),
    (-3, -6, 4, 4),\\
    (-3, -5, 2, 4),
    (-3, -5, 4, 2),\\
    (-2, -5, 1, 2),
    (-2, -5, 2, 1),\\
    (-2, -4, 1, 2),
    (-2, -4, 2, 1),\\
    (-2, -3, 0, 0),
    (-2, -3, 1, 1),\\
    (-2, -3, 1, 2),
    (-2, -3, 2, 1),\\
    (-2, -3, 2, 2),
    (-2, -3, 2, 3),\\
    (-2, -3, 3, 2),
    (-2, -3, 3, 3),\\
    (-1, -3, 0, 0),
    (-1, -2, 0, 0),\\
    (-1, -2, 1, 1)
   \end{array}$ \\
\hline
\end{tabular}
\end{table}

\begin{table}[H] 
\centering\renewcommand{\arraystretch}{1.3}\setlength{\tabcolsep}{2pt}
%\caption{Effective and Nef cone for K3 surfaces of $\varrho(X)=4$.}
%\label{Eff and Nef rank4 5}
\vspace{1cm}
%\footnotesize
\begin{tabular}{|c|c|c|c|c|c|c|}
\hline
 %\multicolumn{5}{ |c| }{ K3 surfaces without elliptic curves} \\
 %\hline
${N}^\circ$&$\NS(X)$& $E(X)$& $\BEff(X)$ & $N(X)$ &  $\BNef(X)$ \\
\hline\hline
13&$V_{13}$ & $\begin{array}{c}
 (0,  0,  0, -1),\\
    (1,  0,  0,  1),\\
    (0,  0, -1,  0),\\
    (1,  1,  0,  0),\\
    (1,  0,  1,  0),\\
    (0, -1,  0,  0)
\end{array}$  &
$E(X)$ &
$\begin{array}{c}
 (2, -1, -1, -1),\\
    (5,  1,  1,  1),\\
    (6, -3, -3,  4),\\
    (6, -3,  4, -3),\\
    (6,  4, -3, -3),\\
    (8, -4,  3,  3),\\
    (8,  3, -4,  3),\\
    (8,  3,  3, -4)
\end{array}$ & See Table \ref{Table2} \\
\hline
14&$V_{14}$ & $\begin{array}{c}
 (1,  0, -1, -2),\\
    (0, -1,  1, -1),\\
    (1,  2, -1,  0),\\
    (2,  2, -3, -2),\\
    (0,  0,  1,  0),\\
    (0,  0, -1,  1),\\
    (1,  1,  0, -2),\\
    (2,  3, -3, -1)
\end{array}$  &
$E(X)$ &
$\begin{array}{c}
( 2,  -3,   2,  -1),\\
    ( 7,  -3,  -8, -11),\\
    ( 7,  12,  -8,   4),\\
    ( 8,   3,   8, -19),\\
    (13,   3,  -2, -29),\\
    (13,  18,  -2, -14),\\
    (17,  12, -28, -16),\\
    (17,  27, -28,  -1),\\
    (22,  27, -38, -11),\\
    (23,  18, -22, -34),\\
    (23,  33, -22, -19),\\
    (28,  33, -32, -29)
    \end{array}$& See Table \ref{Table3}\\
\hline
\end{tabular}
\end{table}

\begin{table}[H] 
\centering\renewcommand{\arraystretch}{1.3}\setlength{\tabcolsep}{2pt}
\caption{$\BNef(X)$ for $\NS(X)=V_1$, $V_{2}$ and $V_{13}$.}
\label{Table2}
\vspace{.2cm}
%\footnotesize
\begin{tabular}{|c|c|}
\hline 
 ${N}^\circ$& $\BNef(X)$ \\
\hline\hline
1&$\begin{array}{c}
  (1, -2, 0, 0),
    (1, -1, 0, -1),
    (1, -1, 0, 0),
    (1, -1, 1, -1),
    (1, -1, 1, 0),
    (1, 0, 0, -2),
    (1, 0, 0, -1),\\
    (1, 0, 0, 0),
    (1, 0, 1, -1),
    (1, 0, 1, 0),
    (1, 0, 2, 0),
    (2, -3, 0, -2),
    (2, -3, 1, -2),\\
    (2, -3, 2, -1),
    (2, -3, 2, 0),
    (2, -2, 0, -3),
    (2, -2, 1, -3),
    (2, -2, 3, -1),
    (2, -2, 3, 0),\\
    (2, -1, 2, -3),
    (2, -1, 3, -2),
    (2, 0, 2, -3),
    (2, 0, 3, -2),
    (3, -5, 2, -2),
    (3, -4, 0, -4),\\
    (3, -4, 1, -4),
    (3, -4, 2, -4),
    (3, -4, 3, -3),
    (3, -4, 4, -2),
    (3, -4, 4, -1),
    (3, -4, 4, 0),\\
    (3, -3, 3, -4),
    (3, -3, 4, -3),
    (3, -2, 2, -5),
    (3, -2, 4, -4),
    (3, -2, 5, -2),
    (3, -1, 4, -4),\\
    (3, 0, 4, -4),
    (4, -6, 3, -4),
    (4, -6, 4, -3),
    (4, -4, 3, -6),
    (4, -4, 6, -3),
    (4, -3, 4, -6),\\
    (4, -3, 6, -4),
    (5, -8, 4, -4),
    (5, -6, 5, -6),
    (5, -6, 6, -5),
    (5, -5, 6, -6),
    (5, -4, 4, -8),\\
    (5, -4, 8, -4),
    (7, -8, 8, -8)
    %son 51 
     \end{array}$\\
     \hline
    2
 & $\begin{array}{c}
    (-3, 6, 4, -4),
    (-3, 6, 8, 4),
    (-2, 4, 3, -2),
    (-2, 4, 5, 2),
    (-1, 1, 0, 0),
    (-1, 2, 1, -1),\\
    (-1, 2, 2, 0),
    (-1, 2, 2, 1),
    (-1, 3, 2, -2),
    (-1, 3, 3, -1),
    (-1, 3, 4, 0),
    (-1, 3, 4, 1),\\
    (-1, 3, 4, 2),
    (0, 1, 0, 0),
    (0, 1, 1, 0),
    (0, 2, 1, -1),
    (0, 2, 2, 1),
    (0, 3, 2, -2),\\
    (0, 3, 3, -1),
    (0, 3, 4, 0),
    (0, 3, 4, 1),
    (0, 3, 4, 2),
    (1, 1, 0, 0),
    (1, 2, 1, -1),\\
    (1, 2, 2, 0),
    (1, 2, 2, 1),
    (1, 3, 2, -2),
    (1, 3, 3, -1),
    (1, 3, 4, 0),
    (1, 3, 4, 1),\\
    (1, 3, 4, 2),
    (2, 4, 3, -2),
    (2, 4, 5, 2),
    (3, 6, 4, -4),
    (3, 6, 8, 4)
    %son 35
\end{array}$\\
\hline
 13& $\begin{array}{c}
 (1, 0, 0, 0),
    (2, -1, -1, -1),
    (2, -1, -1, 0),
    (2, -1, -1, 1),
    (2, -1, 0, -1),
    (2, -1, 0, 0),\\
    (2, -1, 0, 1),
    (2, -1, 1, -1),
    (2, -1, 1, 0),
    (2, 0, -1, -1),
    (2, 0, -1, 0),
    (2, 0, -1, 1),\\
    (2, 0, 0, -1),
    (2, 0, 1, -1),
    (2, 1, -1, -1),
    (2, 1, -1, 0),
    (2, 1, 0, -1),
    (3, -1, 1, 1),\\
    (3, 0, 0, 1),
    (3, 0, 1, 0),
    (3, 1, -1, 1),
    (3, 1, 0, 0),
    (3, 1, 1, -1),
    (4, 0, 1, 1),\\
    (4, 1, 0, 1),
    (4, 1, 1, 0),
    (5, -2, -2, 3),
    (5, -2, 3, -2),
    (5, 1, 1, 1),
    (5, 3, -2, -2),\\
    (6, -3, -3, 4),
    (6, -3, 2, 2),
    (6, -3, 4, -3),
    (6, 2, -3, 2),
    (6, 2, 2, -3),
    (6, 4, -3, -3),\\
    (8, -4, 3, 3),
    (8, 3, -4, 3),
    (8, 3, 3, -4)
  %son 39
 \end{array}$\\
    \hline 
 \end{tabular}
\end{table}

  \begin{table}[H] 
\centering\renewcommand{\arraystretch}{1.3}\setlength{\tabcolsep}{2pt}
\caption{$\BNef(X)$ for $\NS(X)=V_{14}$.}
\label{Table3}
\vspace{.2cm}
%\footnotesize
\begin{tabular}{|c|c|}
\hline 
 ${N}^\circ$& $\BNef(X)$ \\
\hline\hline  
  14 &
   $\begin{array}{c} 
   (1, -1, 0, -1),
    (1, -1, 1, -1),
    (1, 0, -1, -1),
    (1, 0, 0, -2),
    (1, 0, 0, -1),\\
    (1, 0, 0, 0),
    (1, 0, 1, -2),
    (1, 1, -1, -1),
    (1, 1, -1, 0),
    (1, 1, 0, -1),\\
    (2, -3, 2, -1),
    (2, -2, 1, -1),
    (2, -1, -2, -3),
    (2, -1, -1, -3),
    (2, 0, -2, -3),\\
    (2, 0, -1, -4),
    (2, 1, -3, -2),
    (2, 1, -2, -3),
    (2, 1, 0, -4),
    (2, 1, 1, -4),\\
    (2, 2, -3, -1),
    (2, 2, -1, 0),
    (2, 2, 0, -3),
    (2, 3, -3, 0),
    (2, 3, -2, -1),\\
    (2, 3, -2, 0),
    (2, 3, -2, 1),
    (2, 3, -1, -1),
    (3, -1, -3, -5),
    (3, 0, -4, -4),\\
    (3, 1, -1, -6),
    (3, 1, 2, -7),
    (3, 1, 3, -7),
    (3, 2, -2, -5),
    (3, 2, 2, -6),\\
    (3, 3, -5, -2),
    (3, 3, -4, -3),
    (3, 3, -2, -4),
    (3, 4, -5, -1),
    (3, 4, -4, -2),\\
    (3, 4, -2, -3),
    (3, 4, -1, -3),
    (3, 5, -4, 1),
    (3, 5, -3, 1),
    (4, 1, 0, -9),\\
    (4, 1, 1, -9),
    (4, 3, -6, -4),
    (4, 3, -5, -5),
    (4, 4, -4, -5),
    (4, 4, 1, -6),\\
    (4, 5, -4, -4),
    (4, 5, 0, -5),
    (4, 6, -6, -1),
    (4, 6, -5, -2),
    (5, 1, -1, -11),\\
    (5, 2, -7, -6),
    (5, 2, -2, -10),
    (5, 3, -8, -5),
    (5, 4, -5, -7),
    (5, 6, -8, -3),\\
    (5, 6, -7, -4),
    (5, 7, -5, -4),
    (5, 7, -2, -5),
    (5, 7, -1, -5),
    (5, 8, -8, 0),\\
    (5, 8, -7, 0),
    (6, -2, -7, -9),
    (6, 5, -10, -5),
    (6, 7, -7, -6),
    (6, 9, -10, -1),\\
    (6, 10, -7, 3),
    (7, -3, -8, -11),
    (7, 3, 6, -16),
    (7, 4, -5, -12),
    (7, 5, -11, -7),\\
    (7, 5, -6, -11),
    (7, 8, -12, -4),
    (7, 9, -12, -3),
    (7, 10, -6, -6),
    (7, 10, -5, -6),\\
    (7, 11, -11, -1),
    (7, 12, -8, 4),
    (8, 3, 8, -19),
    (8, 7, -8, -11),
    (8, 8, -13, -6),
    (8, 11, -13, -3),\\
    (8, 11, -8, -7),
    (9, 7, -9, -13),
    (9, 10, -10, -10),
    (9, 11, -15, -5),
    (9, 11, -10, -9),\\
    (9, 13, -9, -7),
    (10, 10, -11, -12),
    (10, 13, -11, -9),
    (11, 3, -2, -24),
    (11, 13, -13, -11),\\
    (11, 15, -2, -12),
    (13, 3, -2, -29),
    (13, 18, -2, -14),
    (14, 10, -23, -13),
    (14, 22, -23, -1),\\
    (17, 12, -28, -16),
    (17, 27, -28, -1),
    (18, 22, -31, -9),
    (19, 15, -18, -28),
    (19, 27, -18, -16),\\
    (22, 27, -38, -11),
    (23, 18, -22, -34),
    (23, 27, -26, -24),
    (23, 33, -22, -19),
    (28, 33, -32, -29)
    %son 111
\end{array}$\\
\hline
\end{tabular}
\end{table}

\newpage

\begin{table}[H] 
\centering\renewcommand{\arraystretch}{1.3}\setlength{\tabcolsep}{2pt}
\caption{Intersection matrix of $(-2)$-curves.}
\label{Intmat}
\vspace{.2cm}
%\footnotesize
\begin{tabular}{|c|c|c|}
\hline
${N}^\circ$&Lattice& intersection matrix of $(-2)$-curves \\
\hline\hline
1&$V_1$& 
$\left[\begin{array}{cccccccccccc}
-2&6&0&0&2&4&0&0&4&4&2&4\\
6&-2&4&4&2&0&4&4&0&0&2&0\\
0&4&-2&4&0&6&2&0&4&2&4&0\\
0&4&4&-2&4&0&0&2&2&4&0&6\\
2&2&0&4&-2&4&0&4&0&4&6&0\\
4&0&6&0&4&-2&2&4&0&2&0&4\\
0&4&2&0&0&2&-2&4&0&6&4&4\\
0&4&0&2&4&4&4&-2&6&0&0&2\\
4&0&4&2&0&0&0&6&-2&4&4&2\\
4&0&2&4&4&2&6&0&4&-2&0&0\\
2&2&4&0&6&0&4&0&4&0&-2&4\\
4&0&0&6&0&4&4&2&2&0&4&-2
\end{array}\right]$ \\
\hline
2&$V_2$& 
$\left[\begin{array}{cccccc}
-2&1&6&1&1&1\\
1&-2&1&1&3&0\\
6&1&-2&1&1&1\\
1&1&1&-2&0&3\\
1&3&1&0&-2&1\\
1&0&1&3&1&-2
\end{array}\right]$\\
\hline
3&$V_3$& 
$\left[\begin{array}{ccccc}
-2&0&0&2&1\\
0&-2&2&0&1\\
0&2&-2&0&1\\
2&0&0&-2&1\\
1&1&1&1&-2
\end{array}\right]$ \\
\hline
4&$V_4$& 
$\left[\begin{array}{ccccc}
-2&0&0&2&0\\
0&-2&1&0&2\\
0&1&-2&1&0\\
2&0&1&-2&0\\
0&2&0&0&-2
\end{array}\right]$ \\
\hline
\end{tabular}
\end{table}

\begin{table}[H] 
\centering\renewcommand{\arraystretch}{1.3}\setlength{\tabcolsep}{2pt}
%\caption{Lattice and intersection matrix of $(-2)$-curves of $\varrho(X)=4$.}
%\label{matrix rank4 2}
\vspace{1cm}
%\footnotesize
\begin{tabular}{|c|c|c|}
\hline
${N}^\circ$&Lattice& intersection matrix of $(-2)$-curves \\
\hline\hline
5&$V_5$& 
$\left[\begin{array}{cccccc}
-2&0&2&2&0&0\\
0&-2&0&0&2&2\\
2&0&-2&0&2&0\\
2&0&0&-2&0&2\\
0&2&2&0&-2&0\\
0&2&0&2&0&-2
\end{array}\right]$\\
\hline
6&$V_6$& 
$\left[\begin{array}{cccccccc}
-2&4&2&0&2&6&0&0\\
4&-2&2&6&2&0&0&0\\
2&2&-2&0&1&0&0&3\\
0&6&0&-2&0&4&2&2\\
2&2&1&0&-2&0&3&0\\
6&0&0&4&0&-2&2&2\\
0&0&0&2&3&2&-2&1\\
0&0&3&2&0&2&1&-2
\end{array}\right]$ \\
\hline
7&$V_7$& 
$\left[\begin{array}{cccccccc}
-2&2&2&0&0&2&0&4\\
2&-2&2&4&0&2&0&0\\
2&2&-2&0&4&2&0&0\\
0&4&0&-2&2&0&2&2\\
0&0&4&2&-2&0&2&2\\
2&2&2&0&0&-2&4&0\\
0&0&0&2&2&4&-2&2\\
4&0&0&2&2&0&0&-2
\end{array}\right]$\\
\hline
8&$V_8$& 
$\left[\begin{array}{cccc}
-2&1&0&1\\
1&-2&1&1\\
0&1&-2&0\\
1&1&0&-2
\end{array}\right]$ \\
\hline
9&$V_9$& 
$\left[\begin{array}{cccc}
-2&1&1&1\\
1&-2&0&1\\
1&0&-2&1\\
1&1&1&-2
\end{array}\right]$ \\
\hline
\end{tabular}
\end{table}

\begin{table}[H] 
\centering\renewcommand{\arraystretch}{1.3}\setlength{\tabcolsep}{2pt}
%\caption{Lattice and intersection matrix of $(-2)$-curves of $\varrho(X)=4$.}
%\label{lattice rank4 3}
\vspace{1cm}
%\footnotesize
\begin{tabular}{|c|c|c|}
\hline
${N}^\circ$&Lattice& intersection matrix of $(-2)$-curves \\
\hline\hline
10&$V_{10}$& 
$\left[\begin{array}{cccc}
-2&1&1&1\\
1&-2&1&1\\
1&1&-2&1\\
1&1&1&-2
\end{array}\right]$\\
\hline
11&$V_{11}$& 
$\left[\begin{array}{cccccc}
-2&1&1&1&4&1\\
1&-2&1&4&1&1\\
1&1&-2&1&1&4\\
1&4&1&-2&1&1\\
4&1&1&1&-2&1\\
1&1&4&1&1&-2
\end{array}\right]$ \\
\hline
12&$V_{12}$& 
$\left[\begin{array}{cccccc}
-2&3&0&1&0&1\\
3&-2&1&0&1&0\\
0&1&-2&3&1&0\\
1&0&3&-2&0&1\\
0&1&1&0&-2&3\\
1&0&0&1&3&-2
\end{array}\right]$\\
\hline
13&$V_{13}$& 
$\left[\begin{array}{cccccc}
-2&3&0&1&1&0\\
3&-2&1&0&0&1\\
0&1&-2&1&3&0\\
1&0&1&-2&0&3\\
1&0&3&0&-2&1\\
0&1&0&3&1&-2
\end{array}\right]$ \\
\hline
14&$V_{14}$& 
$\left[\begin{array}{cccccccc}
-2&0&6&0&4&1&1&4\\
0&-2&4&4&0&1&1&6\\
6&4&-2&4&0&1&1&0\\
0&4&4&-2&6&1&1&0\\
4&0&0&6&-2&1&1&4\\
1&1&1&1&1&-2&3&1\\
1&1&1&1&1&3&-2&1\\
4&6&0&0&4&1&1&-2
\end{array}\right]$ \\
\hline
\end{tabular}
\end{table}

\newpage

\begin{center}
\begin{table}[H] 
\centering
\renewcommand{\arraystretch}{1.8}\setlength{\tabcolsep}{2pt}
\caption{ Intersection of a nef and big divisor $H$ with $(-2)$-curves.}
\label{Table5}
\vspace{.2cm}
\begin{tabular}{|c|c|c|c|}
\hline
${N}^\circ$&$\NS(X)$& $H$& Intersection properties\\
\hline\hline

1&$V_1$ & (1,-1,1,-1) &
$H^2=2$, $H\cdot E_i=2$, $i=1,\dots, 12$
\\
\hline

2&$V_2$ & (0,1,1,0)&
{$\!\begin{aligned}
&H^2=2,\, H\cdot E_1=H\cdot E_3=2\\
&H\cdot E_i=1,\, i=2,4,5,6\\
\end{aligned}$}
\\
\hline

3&$V_3$ & (1,0,1,1)&
{$\!\begin{aligned}
&H^2=2,\, H\cdot E_5=0\\
&H\cdot E_i=1, \, i=1,2,3,4\\
\end{aligned}$}
\\ 
\hline

4&$V_4$ & (-2,-2,1,1)&
{$\!\begin{aligned}
&H^2=4,\, H\cdot E_1=H\cdot E_5=2\\
&H\cdot E_i=0,\, i=2,3,4\\
\end{aligned}$}
\\
\hline

5&$V_5$ & (-1,-1,0,1)&
{$\!\begin{aligned}
&H^2=2,\, H\cdot E_i=0,\, i=1,5,6\\
&H\cdot E_i=2,\, i=2,3,4\\
\end{aligned}$}
\\
\hline

6&$V_6$ & (-1,-1,-1,-1)&
{$\!\begin{aligned}
&H^2=2,\, H\cdot E_i=2,\, i=1,2,4,6\\
&H\cdot E_i=1,\, i=3,5,7,8\\
\end{aligned}$}
\\
\hline

7&$V_7$ & (-1,-1,-1,1) &
$H^2=4$, $H\cdot E_i=2$, $i=1,\dots,8$
\\
\hline

8&$V_8$ & (-2,-2,1,0)&
{$\!\begin{aligned}
&H^2=6,\, H\cdot E_1= H\cdot E_4=1\\
&H\cdot E_2= H\cdot E_3=0\\
\end{aligned}$}
\\
\hline

9&$V_9$ & (-1,-1,1,0)&
{$\!\begin{aligned}
&H^2=2,\, H\cdot E_1= H\cdot E_4=1\\
&H\cdot E_2= H\cdot E_3=0\\
\end{aligned}$}
\\ 
\hline

10&$V_{10}$ & (-1,-1,1,1)&
 $H^2=4$, $H\cdot E_i=1$, $i=1,2,3,4$ \\
\hline

11&$V_{11}$ & (-1,-1,2,2)&
 $H^2=4$, $H\cdot E_i=2$, $i=1,\dots, 6$ \\
\hline

12&$V_{12}$ & (1,-2,1,1)&
 $H^2=2$, $H\cdot E_i=1$, $i=1,\dots, 6$ \\
\hline

13&$V_{13}$ & (1,0,0,0)&
 $H^2=2$, $H\cdot E_i=1$, $i=1,\dots, 6$ \\
\hline

14&$V_{14}$ & (1,1,-1,-1)&
{$\!\begin{aligned}
&H^2=2,\, H\cdot E_i= 2,\, i=1,2,3,4,5,8\\
&H\cdot E_6= H\cdot E_7=1\\
\end{aligned}$}
\\
\hline
\end{tabular}
\end{table}
\end{center}
\restoregeometry

%\newgeometry{left=3cm,bottom=0.1cm}
 \begin{center}
\begin{table}[H]
\centering
\renewcommand{\arraystretch}{1.8}\setlength{\tabcolsep}{2pt}
\caption{Degrees of a set of generators of $R(X)$.}\label{TableGen rank4}
\vspace{.2cm}
\begin{tabular}{|c|c|l|}
% \hline
  %  \multicolumn{3}{ |c| }{ K3 surface without elliptic curves} \\
  \hline
${N}^\circ$&$\NS(X)$&  Degrees of generators of $R(X)$\\
\hline
\hline
1&$V_1$&$\BEff$\\
\hline
2&$V_2$&$E, \BNef[i],i=1-5, 7, 11, 14, 15, 20, 23, 25, 29, 32-35$\\
\hline
3&$V_3$&$E, \BNef[i],i=1, 2, 4, 7, 9$\\
\hline
4&$V_4$&$E, \BNef[3], \BNef[3]+\BNef[4]$\\ 
\hline
5&$V_5$&$E, \BNef[1]+\BNef[2] $\\
\hline
6&$V_6$& 
 $E, \BNef[15]$\\
\hline
7&$V_{7}$&$E$\\
%$E, \BNef[11]^* $\\
\hline
8&$V_{8}$&{$\!\begin{aligned}
&E, \BNef[i], i=1,2,3,5 \\ 
%&\BNef[4]^{*}
\end{aligned}$}\\
\hline
9&$V_{9}$&$E, \BNef[i],i=1-4, 6, 7$\\
\hline
10&$V_{10}$&$E, \BNef[i],i=2-5$\\
\hline
11&$V_{11}$&
{$\!\begin{aligned}
&E, \BNef[i],i=4, 7, 9, 11, 13, 15, 16, 18, 20, 21, 23, 24, 26, 27\\
&\BNef[i]^*, i=1-3, 5, 6, 8, 10, 12, 14, 17, 19, 22, 25
\end{aligned}$}\\
\hline
12&$V_{12}$&$E, \BNef[30], \BNef[31]$\\
\hline
13&$V_{13}$&$E,\BNef[i],i=1^*, 7, 9, 12, 14, 16, 17, 27, 28, 30-39$\\
  \hline
14&$V_{14}$& Contains the degrees in Table \ref{Table degree generator F14}\\
\hline
\end{tabular}
\end{table}
\end{center}

 \begin{table}[H] 
\centering\renewcommand{\arraystretch}{1.3}\setlength{\tabcolsep}{2pt}
\caption{Degrees of generators of the Cox ring of the family $\mathcal F_{14}$.}
\label{Table degree generator F14}
\vspace{.2cm}
%\footnotesize
\begin{tabular}{|c|c|}
\hline 
 ${N}^\circ$& degrees of generators\\
\hline\hline  
  14 &
   $\begin{array}{c} 
   (1, 0, -1, -2 ),
    ( 0, -1, 1, -1 ),
    ( 1, 2, -1, 0 ),
    ( 2, 2, -3, -2 ),
    ( 0, 0, 1, 0 ),\\
    ( 0, 0, -1, 1 ),
    ( 1, 1, 0, -2 ),
    ( 2, 3, -3, -1),
    (1, -1, 0, -1),
    (1, 0, -1, -1),\\
    (1, 0, 0, 0),
    (1, 1, -1, -1),
    (1, 1, -1, 0),
    (2, -3, 2, -1),
    (2, -2, 1, -1),\\
    (2, -1, -2, -3),
    (2, 1, -3, -2),
    (2, 1, 0, -4),
    (2, 2, -3, -1),
    (2, 2, 0, -3),\\
    (2, 3, -3, 0),
    (2, 3, -2, 1),
    (3, 0, -4, -4),
    (3, 1, 2, -7),
    (3, 2, -2, -5),\\
    (3, 2, 2, -6),
    (3, 3, -5, -2),
    (3, 3, -2, -4),
    (3, 4, -5, -1),
    (3, 4, -2, -3),\\
    (3, 5, -4, 1),
    (4, 1, 0, -9),
    (4, 4, -4, -5),
    (4, 5, -4, -4),
    (4, 5, 0, -5),\\
    (5, 2, -2, -10),
    (5, 3, -8, -5),
    (5, 7, -2, -5),
    (5, 8, -8, 0),
    (6, -2, -7, -9),\\
    (6, 5, -10, -5),
    (6, 9, -10, -1),
    (6, 10, -7, 3),
    (7, -3, -8, -11),
    (7, 3, 6, -16),\\
    (7, 5, -6, -11),
    (7, 8, -12, -4),
    (7, 9, -12, -3),
    (7, 10, -6, -6),
    (7, 12, -8, 4),\\
    (8, 3, 8, -19),
    (8, 7, -8, -11),
    (8, 11, -8, -7),
    (9, 10, -10, -10),
    (9, 11, -10, -9),\\
    (11, 3, -2, -24),
    (11, 15, -2, -12),
    (13, 3, -2, -29),
    (13, 18, -2, -14),\\
    (14, 10, -23, -13),
    (14, 22, -23, -1),
    (17, 12, -28, -16),
    (17, 27, -28, -1),\\
    (18, 22, -31, -9),
    (19, 15, -18, -28),
    (19, 27, -18, -16),
    (22, 27, -38, -11),\\
    (23, 18, -22, -34),
    (23, 27, -26, -24),
    (23, 33, -22, -19),
    (28, 33, -32, -29)
\end{array}$\\
\hline
\end{tabular}
\end{table}
\restoregeometry

\bibliography{ref}
\bibliographystyle{alpha}

\end{document}